\documentclass[final,onefignum,onetabnum]{siamonline190516}

\usepackage{braket,amsfonts}

\usepackage{array}

\usepackage{pgfplots}
\usepackage{booktabs} 

\usepackage[english]{babel}
\usepackage[utf8x]{inputenc}
\usepackage[T1]{fontenc}
\usepackage{curves} 
\usepackage{algorithm}
\usepackage{algorithmic}
\usepackage{amsmath}
\usepackage{amsfonts}
\usepackage{graphicx}
\usepackage{caption}
\usepackage{subcaption}

\usepackage[colorinlistoftodos]{todonotes}
\usepackage{pgfplots}

\Crefname{ALC@unique}{Line}{Lines}

\newsiamremark{assumption}{Assumption}
\newsiamremark {remark}{Remark}
\newsiamremark {example}{Example}
\newsiamthm {problem}{Inverse Problem}
\newsiamthm {conjecture}{Conjecture}

\usepackage{amsopn}

\newcommand{\R}{\mathbb{R}}

\newcommand{\iter}[2]{{#1}_{#2}}
\newcommand{\trni}[2]{{#1}^{#2}}
\newcommand{\ittrni}[3]{{#1}^{#2}_{#3}}

\newcommand{\range}[1]{{\rm \bf rng}({#1})}
\newcommand{\kernel}[1]{{\rm \bf kern}({#1})}
\newcommand{\dom}[1]{{\rm \bf dom}({#1})}

\DeclareMathOperator*{\argmin}{arg\,min}

\newcommand{\dnabla}{\nabla^\dagger}
\newcommand{\TrueOp}{\ensuremath{A}}
\newcommand{\ApproxOp}{\ensuremath{\widetilde{A}}}
\newcommand{\CorrectedOp}{\ensuremath{A_\Theta}}
\newcommand{\CorrectedAd}{\ensuremath{A^*_\Phi}}
\newcommand{\ForwardCor}{\ensuremath{F_\Theta}}
\newcommand{\AdjointCor}{\ensuremath{G_\Phi}}

\newcommand{\RegularisationOp}{\ensuremath{R}}

\newcommand{\revision}[1]{#1}
\newcommand{\revisionTwo}[1]{{#1}}

\title{On Learned Operator Correction in Inverse Problems\thanks{SL and AH contributed equally.\funding{
			Academy of Finland Projects 312123 and 312342 (Finnish Centre of Excellence in Inverse Modelling and Imaging, 2018--2025) and Projects 334817 and 314411,
			 Jane and Aatos Erkko Foundation,
			  British Heart Foundation grant NH/18/1/33511,
			  CMIC-EPSRC platform grant (EP/M020533/1),
			  EPSRC-Wellcome grant WT101957,
		 Leverhulme Trust project on Breaking the non-convexity barrier,
		 Philip Leverhulme Prize,
		 EPSRC grants EP/S026045/1, EP/T003553/1, EP/T000864/1, 
		 EPSRC Centre Nr. EP/N014588/1,
		 Wellcome Innovator Award RG98755,
		 the RISE projectsCHiPS and NoMADS,
		 Cantab Capital Institute for the Mathematics of Information,
		 Alan Turing Institute,
		 EPSRC grant EP/L016516/1 for the University of Cambridge Centre for Doctoral Training (Cambridge Centre for Analysis)
}}}
\author{Sebastian Lunz\thanks{University of Cambridge, Department of Applied Mathematics and Theoretical Physics (\email{lunz@math.cam.ac.uk}).}
	\and Andreas Hauptmann\thanks{University of Oulu, Research Unit of Mathematical Sciences; University College London, Department of Computer Science (\email{Andreas.Hauptmann@oulu.fi}). }
	\and Tanja Tarvainen\thanks{University of Eastern Finland, Department of Applied Physics; University College London, Department of Computer Science}
	\and Carola-Bibiane Sch\"onlieb\footnotemark[2]
	\and Simon Arridge\thanks{University College London, Department of Computer Science}
}

\headers{On Learned Operator Correction}{S. Lunz, A. Hauptmann, T. Tarvainen, C.-B. Sch\"onlieb, S. Arridge}

\begin{document}
\maketitle
\begin{abstract}
	We discuss the possibility to learn a data-driven explicit model correction for inverse problems and whether such a model correction can be used within a variational framework to obtain regularised reconstructions. This paper discusses the conceptual difficulty to learn such a forward model correction and proceeds to present a possible solution as forward-adjoint correction that explicitly corrects in both data and solution spaces. We then derive conditions under which solutions to the variational problem with a learned correction converge to solutions obtained with the correct operator. The proposed approach is evaluated on an application to limited view photoacoustic tomography and compared to the established framework of Bayesian approximation error method.
\end{abstract}


\section{Introduction}
In inverse problems 
it is 
usually considered imperative to have an accurate forward model of the underlying physics. Nevertheless, such accurate models can be computationally highly expensive due to possible nonlinearities, large spatial and temporal dimensions as well as stochasticity. Thus, in many applications approximate models are used in order to speed up reconstruction times and to comply with hardware and cost restrictions. As a consequence the introduced approximation errors need to be taken into account when solving ill-posed inverse problems or a degradation of the reconstruction quality can be expected.

For instance, in classical computerised tomography with a relatively high dose, models based on ray transforms are sufficiently accurate for the 
reconstruction task, whereas the full physical model would incorporate stochastic X-ray scattering events. 
Nevertheless, in some cone beam computerised tomography applications the dose is typically relatively low with a large field of view and hence scattering 
becomes more prevalent \cite{siewerdsen2001cone} and simple models based on the 
ray transform are not enough to guarantee sufficient image quality. However, as these scattering events are stochastic, accurate models would be too expensive for practical image reconstruction. Therefore, the basic model is used as approximation with an appropriate correction that accounts for the full physical phenomena \cite{zhu2009scatter}. 

In applications where the forward model is given by the solution of a partial 
differential equation, model reduction techniques are often used to reduce 
computational cost~\cite{borcea2020ModelReduction,Freund2003ModelReduction,smyl2019less}. Such reductions lead to known approximation errors in the model and can 
be corrected for by explicit modelling \cite{arridge2006a,Kaipio2004a}. 
Recently, with the 
possibility to combine deep learning techniques with classical variational
methods, approximate models are now also used in the framework of learned image 
reconstruction \cite{hauptmann2018MLMIR}. In this case, the approximate model is 
embedded in an iterative scheme and updates are performed by a convolutional
neural network (CNN). Here, model correction is performed \emph{implicitly} by 
the network while computing the iterative updates.

In this paper we investigate the possibility to correct such approximation errors \emph{explicitly} with data-driven methods, in particular 
using a CNN. In what follows, we restrict ourselves to linear inverse problems, \revision{with both theory and experiments considering the linear case only. However, we expect many of the challenges and approaches discussed here to be relevant and to give insight into the non-linear case as well.} 
Let $x\in X$ be the unknown quantity of interest we aim to reconstruct from measurements $y\in Y$, where $X$, $Y$ are Hilbert spaces 
and $x$ and $y$ fulfil the relation
\begin{equation}\label{eqn:invProb}
Ax=y,
\end{equation}
where $A:X\to Y$ is the accurate forward operator modelling the underlying physics \revision{sufficiently accurate for any systematic error to be well below the noise level of the acquisition}.
We assume that the evaluation of accurate operator $A$ is computationally expensive and we 
rather want to use an approximate model $\ApproxOp:X\to Y$ to compute $x$ from $y$. In doing so, we 
introduce an inherent approximation error in \eqref{eqn:invProb} and 
have 
\begin{equation}\label{eqn:invProb_approxModel}
\ApproxOp x=\widetilde{y}. 
\end{equation}
leading to a systematic model error 
\begin{equation}\label{eqn:SystematicApproximationError}
\delta y = y - \widetilde{y}.
\end{equation}
\begin{remark}\label{remark:ApOpRangeDomain}
	In general, the range and domain of $\ApproxOp$ might be different to those of $A$. To simplify the remainder of this paper we assume, unless otherwise stated, that appropriate projections between the range and domain of the approximate operator $\ApproxOp$ as well as the range and domain of the accurate operator $\TrueOp$ 
	are included in the implementation of $\ApproxOp$, so that expressions such as \eqref{eqn:SystematicApproximationError} are well defined. 
\end{remark}



In this work, we consider corrections for this approximation error via a 
parameterisable, possibly nonlinear, mapping $\ForwardCor:Y\to Y$, applied as a correction to $\ApproxOp$. This leads to a corrected operator $\CorrectedOp$ of the form
\begin{equation}\label{eqn:correctedModel}
\CorrectedOp=\ForwardCor\circ\ApproxOp.
\end{equation}
We aim to choose the correction $\ForwardCor$ such that ideally $\CorrectedOp(x) \approx \TrueOp x$ for some $x\in X$ of interest. Restricting the corrected operator $\CorrectedOp$ to be a composition 
of the approximate operator $\ApproxOp$ and a parameterisable correction yields 
various advantages compared to fully parameterising the corrected operator $\CorrectedOp: X \to Y$, without utilising the knowledge of $\ApproxOp$. 
It avoids having to model the typically global dependencies of $\TrueOp$ in the 
learned correction and allows us to employ generic 
network architectures for $\ForwardCor$, such as the popular U-Net \cite{Ronneberger2015}.

The primary question that we aim to answer is, whether such corrected 
models \eqref{eqn:correctedModel} can be subsequently used in variational regularisation
approaches that find a reconstruction $x^*$ as
\begin{equation}\label{eqn:varProbIntro}
x^* = \argmin_{x\in X} \frac{1}{2}\|\CorrectedOp(x) - y\|_Y^2 + \lambda \RegularisationOp(x)
\end{equation}
with regularisation functional $\RegularisationOp$ and associated 
hyper-parameter $\lambda$.  Apart from investigating the practical performance of \eqref{eqn:varProbIntro}, we will discuss conditions on the model correction that need to be satisfied to guarantee convergence of solutions to \eqref{eqn:varProbIntro} to the accurate solution \revision{as the corrected operator $\CorrectedOp$ approaches the accurate operator $\TrueOp$}. We provide theoretical results, which show that variational regularisation strategies can be applied under certain conditions.
In particular, as we will discuss in this study, while it is fairly easy to learn a model correction that fulfils \eqref{eqn:correctedModel}, it cannot be readily guaranteed to yield high-quality reconstructions when used within the variational problem \eqref{eqn:varProbIntro}.
This is a conceptual difficulty caused by a possible discrepancy in the range of the adjoints of $A$ and $\ApproxOp$ that can be an inherent part of the approximate model and hence first order methods to solve \eqref{eqn:varProbIntro} yield non-desirable results. 


To overcome this restriction, we introduce a 
forward-adjoint correction that combines an explicit forward model correction 
with an explicit correction of the adjoint. We will show that such a 
forward-adjoint correction -if trained sufficiently well- provides a descent direction for a gradient 
scheme to solve \eqref{eqn:varProbIntro} for which we can guarantee convergence to a neighbourhood of the solution obtained with the accurate operator $\TrueOp$. 


This work fits into the wider field of learned image reconstruction techniques 
that have sparked large interest in recent years \cite{arridge2019solving,Jin2017,Kang2017}. In particular, we are motivated by 
model-based learned iterative reconstruction techniques that have shown to be highly 
successful in a variety of application areas \cite{Adler2017,Adler2018,hammernik2018learning,Hauptmann2018, schlemper2017deep}.  
These methods generally mimic iterative gradient descent schemes and demonstrate
impressive reconstruction results with often considerable speed ups \cite{hauptmann2019multi}, 
but are mostly empirically motivated and lack convergence guarantees. In contrast,
this paper follows a recent development of understanding how deep learning methods 
can be combined with classical reconstruction algorithms, such as variational 
techniques, and retaining established theoretical results on convergence. Whereas 
most studies are concentrated on learning a regulariser  \cite{kobler2020total,li2020nett, lunz2018adversarial, schwab2019deep}, we here 
concentrate on the operator only and keep a fixed, analytical form for the regulariser. Further, related \revision{works that consider} learned corrections by utilising explicit knowledge of the operator range are \cite{boink2019learned,bubba2019learning,schwab2019deep}.
Another line of research examines the incorporation of imperfectly known forward operators in a fully variational model \cite{burger2019convergence,korolev2018image}
as well as perturbations in \cite{egger2015identification,lorz2019parameter}. We note also the connection to the concept of calibration in a Bayesian setting \cite{kennedy2001}. 

This paper is organised as follows. In Section \ref{sec:learnModelCor}, we introduce the concept of model correction and compare to previous work in the field. In Section \ref{sec:forwardCorrection}, we discuss forward corrections and demonstrate their limitations. To overcome these limitations, we introduce in the following the forward-adjoint corrections in Section \ref{sect:ForwardBackwardCorr}, where we also present convergence results for this correction. This is followed by a discussion of computational challenges and the experimental setup in Section \ref{sec:ComputationalConsiderations}. Finally, in Section \ref{sec:computationalResults}, we  demonstrate the performance of the discussed approaches on two data sets for limited-view photoacoustic tomography.

\subsection*{Glossary}
To improve readability throughout the paper we provide a Glossary with the definition of frequently used notation.

\begin{table}[ht!]
	\begin{tabular}{|l|l|l|}\hline
		Symbol & Description & Definition \\ \hline
		$ X$ & Reconstruction space & Hilbert Space, Norm $\| \cdot \|_X$, Product $\langle \cdot, \cdot \rangle_X$ \\
		$ Y$ & Measurement space & Hilbert Space, Norm $\| \cdot \|_Y$, Product $\langle \cdot, \cdot \rangle_Y$ \\ 
		$ A$ & Exact forward operator & $A:X\to Y$ \\ 
		$\ApproxOp$ & Approximate forward operator & $\ApproxOp :X\to Y$  \\ 
		$\ForwardCor$ & Parameterisable correction in $Y$ & $\ForwardCor:Y\to Y$ \\ 
		$\AdjointCor$ & Parameterisable correction in $X$& $\AdjointCor:X\to X$  \\
		$\CorrectedOp$ & Corrected forward operator  & $\CorrectedOp: X \to Y$, $\CorrectedOp = \ForwardCor \circ \ApproxOp$  \\
		$\CorrectedAd$ & Corrected adjoint & $\CorrectedAd: Y \to X$, $\CorrectedAd = \AdjointCor \circ \ApproxOp^*$ \\
		$Df(t)$ &Fr{\'{e}}chet derivative of $f$ at $t$ & $Df(t) :\dom f \to \range f$\\
		&&$f(t+\delta t) = f(t) +Df(t)\delta t + \mathcal{O}(\delta t^2) $ \\
		$R$ & Regularisation functional & $R :X \to \R_+$ \\
		$\mathcal{L}$ & Variational functional with $\TrueOp$ & $\mathcal{L}(x) = \frac{1}{2}\|\TrueOp x - y\|^2_Y + \lambda R(x)$  \\
		$\mathcal{L}_\Theta$ & Variational functional with $\CorrectedOp$ & $\mathcal{L}_\Theta(x) =\frac{1}{2} \|\CorrectedOp(x) - y\|^2_Y + \lambda R(x)$  \\
		& & 
		\\\hline
	\end{tabular}
	\caption{ }
\end{table}


\section{Learning a model correction}\label{sec:learnModelCor}

As we have motivated above, we only consider an explicit model correction \eqref{eqn:correctedModel} in this study and leave the regularisation term untouched. Therefore, we will discuss in the following how a model correction using data driven methods is possible and what the main challenges are. 

Before we turn to the discussion of an \emph{explicit} correction, it is important to make the distinction to an \emph{implicit} correction in the framework of learned iterative reconstructions. In particular, we concentrate here on learned gradient schemes \cite{Adler2017}, which can be formulated by a network $\Lambda_\Theta$, that is 
designed to mimic a gradient descent step. In particular, we train the networks to perform an iterative update, such that
\begin{equation}\label{eqn:LGS}
\iter{x}{k+1}=\Lambda_\Theta \left(\nabla_x \frac{1}{2}\|A\iter{x}{k}-y\|_Y^2,\iter{x}{k}\right),
\end{equation}
where $\nabla_x \frac{1}{2}\|A\iter{x}{k}-y\|_Y^2 = A^*(A\iter{x}{k}-y)$. Now, one could use an 
approximate model instead of the accurate model and compute an approximate 
gradient given by $\ApproxOp^*(\ApproxOp \iter{x}{k}-y)$ for the update in \eqref{eqn:LGS}, as proposed in \cite{hauptmann2018MLMIR}. The network $\Lambda_\Theta$ then \emph{implicitly} corrects the model error to produce the new iterate.
That means, the correction and a prior are hence trained simultaneously with the update in \eqref{eqn:LGS}. Such approaches are typically trained by using a loss function, like the $L^2$-loss, to measure the distance between reconstruction and a ground truth phantom.

On the other hand, in the \emph{explicit} approach that we pursue here, we aim to learn a correction $\CorrectedOp$ that is independent of the regularisation use. It can hence be trained using knowledge of the accurate and approximate operator alongside training data in either $X$ or $Y$, without requiring pairs of measurements and their corresponding ground truth phantoms. In a scenario where the operators cannot been accessed directly, samples of pairs from the two operators can even be sufficient to fit an explicit operator correction. While implicit methods have been shown to perform well in practice \cite{hauptmann2018MLMIR}, our approach will yield an explicit correction and as such can be used in combination with any \revision{regularisation functional} and builds on the established variational framework. Furthermore, we note that the study of explicit methods also allows to uncover and investigate some of the fundamental challenges of model correction that might easily be left ignored in implicit approaches.


Thus, we will concentrate our discussion in the following on how an \emph{explicit} data correction can be achieved, how the correction of the model $\ApproxOp$ can be parametrised by a neural network, and how this can be incorporated into a variational framework.


\subsection{Approximation error method (AEM)\label{sect:AEM}}

A well-established approach to incorporate model correction into a reconstruction framework, such as \eqref{eqn:varProbIntro}, is given by Bayesian approximation error modelling \cite{Kaipio2004a,Kaipio2007}. 
Let us shortly recall, that in Bayesian inversion we want to determine the posterior distribution of the unknown $x$ given $y$, and by Bayes' formula we obtain 
\begin{equation}\label{eqn:Bayes}
p(x|y) = p(y|x) \frac{p(x)}{p(y)}.
\end{equation}
Thus, the posterior distribution is characterised by the likelihood $p(y|x)$ and the chosen prior $p(x)$ on the unknown. Typically, the likelihood $p(y|x)$ is modelled using accurate knowledge of the forward operator $A: X \to Y$ as well as the noise model. In the approximation error method, the purpose is now to adjust the likelihood by examining the difference between the (accurate) forward operator $A$ and its approximation $\ApproxOp$  of the model  \eqref{eqn:invProb}--\eqref{eqn:invProb_approxModel} as
\begin{equation}\label{eqn:modelErrorBAE}
\varepsilon  =  \delta y = Ax-\ApproxOp x. 
\end{equation}
Including an additive model for the measurement noise $e$, this leads to an observation model  
\begin{equation}\label{eqn:observationModelBAE}
y =  \ApproxOp x + \varepsilon + e.
\end{equation} 
We model the noise $e$ independent of $x$ as Gaussian $e \sim \mathcal{N}(\eta_e,\Gamma_e)$,
where $\eta_e$ and $\Gamma_e$ are the mean and covariance of the noise.
Further, the model error $\varepsilon$ is approximated as Gaussian $\varepsilon \sim \mathcal{N}(\eta_{\varepsilon},\Gamma_{\varepsilon})$ and is modelled independent of noise $e$ and unknown parameters $x$  leading to a Gaussian distributed total error $n=\varepsilon+e$, $n \sim \mathcal{N}(\eta_n,\Gamma_n)$,  where $\eta_{\varepsilon}$ and $\eta_n$ are means and $\Gamma_{\varepsilon}$ and $\Gamma_n$ are the covariance matrices of model error and total errors, respectively.
This leads to a so-called enhanced error model \cite{Kaipio2004a} with a likelihood distribution of the form
\[
p(y|x) \sim \exp \left( - \frac{1}{2}\| L_n(\ApproxOp x - y + \eta_n) \|_Y^2 \right)
\]
where $L_n^{\rm T}L_n=\Gamma_n^{-1}$  is a matrix square root such as the Cholesky decomposition of the inverse covariance matrix of the total error.
In the case of Gaussian white noise with  a zero mean and a constant standard deviation $\sigma$, this can be written as
\[
p(y|x) \sim \exp \left( - \frac{1}{2\sigma}\| L_{\varepsilon}(\ApproxOp x - y + \eta_{\varepsilon}) \|_Y^2 \right)
\]
where $L_{\varepsilon}^{\rm T}L_{\varepsilon}=\Gamma_{\varepsilon}^{-1}$. 
This could be used to motivate writing the variational problem \eqref{eqn:varProbIntro} in a form
\begin{equation}\label{eqn:varProbBAE}
x^* = \argmin_{x\in X} \frac{1}{2}\|L_{\varepsilon} (\ApproxOp x - y + \eta_{\varepsilon}) \|_{\revision{Y}}^2 + \lambda \RegularisationOp(x).
\end{equation}

In order to utilise the approach, the unknown distribution of the model error 
needs to be approximated. 
That can be obtained for example by simulations \cite{arridge2006a,tarvainen2013} as follows. 
Let $\lbrace \trni{x}{i}, i = 1,\ldots , N \rbrace$ be a set of samples drawn from a training set.  
The corresponding samples of the model error are then
\begin{equation}
\trni{\varepsilon}{i} = A\trni{x}{i}-\ApproxOp\trni{x}{i} 
\end{equation}
and the mean and covariance of the model error can be estimated from the samples as 
\begin{align}
\eta_{\varepsilon} & = \frac{1}{N} \sum_{i = 1}^{N} \trni{\varepsilon}{i} \\
\Gamma_{\varepsilon} & = \frac{1}{N-1} \sum_{i = 1}^{N} \trni{\varepsilon}{i}(\trni{\varepsilon}{i})^{\rm T} -\eta_{\varepsilon}\eta_{\varepsilon}^{\rm T}.
\end{align}


\subsection{Learning a general model correction}\label{sec:learnedDataCor}

The classical Bayesian approximation error method provides an affine linear correction of the likelihood in \eqref{eqn:varProbBAE} and by construction is limited to cases where the error between accurate and approximate model  \eqref{eqn:modelErrorBAE} can be approximated as normally distributed. As this can be too restrictive in certain cases to describe more complicated errors, we will now address a more general concept of learning a  nonlinear \emph{explicit} model correction.

That is, given an accurate underlying forward model $\TrueOp$, we aim to find a 
(partially) learned operator $\CorrectedOp$ which we consider as an explicitly corrected approximate model of the form~(\ref{eqn:correctedModel}). 
To do so, we need to set a notion of distance between $\TrueOp$ and $\CorrectedOp$ in order to assess the quality of the approximation. A seemingly natural notion of distance between two operators would be the supremum norm over elements in $X$, that is we consider here
\begin{align}
\label{equ:UniformError}
\|\TrueOp - \CorrectedOp\|_{X\to Y} &:= \sup_{x\in X : \|x\|=1} \|\TrueOp x - \CorrectedOp(x)\|_Y.
\end{align}
However, in many relevant applications it is impossible to find a correction of the form $\CorrectedOp = \ForwardCor \circ \ApproxOp$ that achieves low uniform approximation error, making this notion of distance too restrictive.
For instance, if we consider the case of a learned \emph{a-posteriori} correction of some approximate model $\ApproxOp$ with a parameterisable mapping $\ForwardCor:Y\to Y$ that fulfils \eqref{eqn:correctedModel},
then the approximate model $\ApproxOp$ can exhibit a nullspace $\kernel{\ApproxOp}$ that is different from that of the accurate operator and, in particular, is potentially much larger. Thus, there may exist a (or several) $v \in \kernel{\ApproxOp}$ with $\TrueOp v \neq 0$. Any corrected operator $\CorrectedOp = \ForwardCor \circ \ApproxOp$ 
then exhibits an error in the sense of \eqref{equ:UniformError} of at least $\|\TrueOp v\|_Y$, as 
\begin{align*}
\| \TrueOp - \CorrectedOp \|_{X\to Y} &\geq \max\{  \| \TrueOp v - \ForwardCor(0) \|_Y, \| \TrueOp(-v) - \ForwardCor(0) \|_Y \} \\
&\geq \min_{y \in Y} \max \{  \| \TrueOp v - y \|_Y, \| - \TrueOp v - y \|_Y \} \\
&= \| \TrueOp v \|_Y,
\end{align*}
\revision{where in the last equality we have used that the point minimising the maximum of the distance to other two points is the centre of the line through those points. In our case, the centre of the line between $\TrueOp v$ and $-\TrueOp v$ is always the origin of the coordinate system $0$, independently of the choice of $\TrueOp$ and $v$.} 
In other words, the information in direction $v$ is lost in the approximate model and would need to be recovered subsequently by the correction $\ForwardCor$. If there are several such non-trivial $v \in \kernel{\ApproxOp}$, a uniform correction becomes increasingly difficult in the form of \eqref{equ:UniformError}.
We will illustrate this difficulty in the following Section \ref{Sec:DownsamplingToyCase}.
\\
\\
While aiming for a uniform correction is unpractical, it can nevertheless be possible to correct the operator $\ApproxOp$ using an \emph{a-posteriori} correction as in \eqref{eqn:correctedModel}, provided a weaker notion of operator distance is employed. Here, we propose an empirical, learned notion of operator correction, that is optimised for a training set of points $\{\trni{x}{i}, i = 1,\ldots,N\}$, similar to Section \ref{sect:AEM}. 
More precisely, we examine the average deviation of $\CorrectedOp$ from $\TrueOp$ as
\begin{align}
\label{equ:StatisticalCorrection}
\frac{1}{N} \sum_{i}^N \|\CorrectedOp(\trni{x}{i}) - \TrueOp\trni{x}{i}\|_Y,
\end{align}
in a suitable norm $\|\cdot\|.$
In this notion, it is sufficient for the operators to be close in the mean for a given training set and hence we call this a \emph{statistical} or \emph{learned} correction with respect to the chosen training set. 
For instance,  if the kernel direction $v \in \kernel{\ApproxOp}$ is orthogonal to the sample $x^i$, the information lost in direction $v$ is not crucial for representing the data of interest. Alternatively, 
the kernel direction $v$ might be highly correlated with another direction $w\notin \kernel{\ApproxOp}$ in the sense that $\langle \trni{x}{i} , v \rangle \approx \langle \trni{x}{i}, w \rangle$ for all $i$. Then the result of $\TrueOp v$ can be inferred from $\ApproxOp w$, even though $\ApproxOp v = 0$.
\\
To conclude this section, we note that in many cases we cannot hope to find a uniform model correction, but that correcting the model error can be still attempted using the notion of learned correction, quantified by \eqref{equ:StatisticalCorrection}. This is possible even if the operators $\TrueOp$ and $\ApproxOp$ are exhibiting different kernel spaces, as long as the training set $\{\trni{x}{i}, i = 1,\ldots,N\}$ exhibits sufficient structure to compensate for the loss of information in the approximate model.

\begin{remark}
	We consider non-linear corrections $\CorrectedOp = F_\Theta \circ \ApproxOp$ in this paper even when correcting a linear operator $\TrueOp$ from a linear approximation $\ApproxOp$, as in our computational examples. We have three main motivations to do so. Firstly, there are well-established nonlinear  network architectures, such as U-Net \cite{Ronneberger2015}, that are highly powerful and in fact have considerably fewer parameters than a fully parametrised linear map when the method is applied to applications in 3D, making the non-linear approach scalable. Secondly, when considering nonlinear corrections, a generalisation to the context of nonlinear operators will be easier. Finally and most importantly, while the operators $\TrueOp$ and $\ApproxOp$ might be linear, the region of interest in image and data space where we need a good correction is highly nonlinear, in the sense that the samples $\trni{x}{i}$ in \eqref{equ:StatisticalCorrection} are drawn from a distribution with nonlinear support. This makes nonlinear corrections considerably more powerful in correcting model errors than their linear counterparts.
\end{remark}

\subsubsection{A toy case: downsampling}
\label{Sec:DownsamplingToyCase}
In order to illustrate the challenge of a learned operator correction, we consider a toy case. Here, the accurate forward model $\TrueOp$ is given by a downsampling operator with an averaging filter, while the approximate model $\ApproxOp$ simply skips every other 
sample. Concretely, we consider $x\in \R^n$, $y\in\R^{n/2}$ and $\ApproxOp,A\in \R^{n/2\times n}$, given by
\begin{equation}\label{equ:downsamplingOps}
A=\begin{pmatrix} 
\frac{1}{2}&\frac{1}{4} &                         &  &  \\
&\frac{1}{4} & \frac{1}{2}&\frac{1}{4}  \\
&  & \ddots &\ddots & \ddots &  \\
&  &  & \frac{1}{4} & \frac{1}{2} & \frac{1}{4} 
\end{pmatrix}, \text{ and }  
\ApproxOp=\begin{pmatrix} 
1& 0 &   &  &  \\
& 0& 1&0  \\
&  & \ddots &\ddots & \ddots &  \\
&  &  & 0 & 1 & 0 
\end{pmatrix}.
\end{equation}

Clearly, both operators have very different kernel spaces, with $\TrueOp$ vanishing on inputs \revision{of even magnitude} with alternating sign, whereas $\ApproxOp$ vanishes for every $v$ with $v[j]=0$, with index $j$ \revision{\emph{even}}, and any value for $j$ \revision{\emph{odd}}. In other words, the null space is spanned by the unit vectors with \revision{odd} index, $\kernel{\ApproxOp}=\{\mathbf{e}_j\ | \ 0<j\leq n, \ j \ \revision{\textit{even}} \}$.
In fact, by the same argument as above, these $v\in \kernel{\ApproxOp}$ with $\|v\|_\infty = 1$ are such that the uniform approximation error for any correction will be $\|\TrueOp v - \ForwardCor( \ApproxOp v)\|_{\infty} \geq \|A v\|_\infty \geq 0.25 $ for all $v\in \kernel{\ApproxOp}$.

This example exhibits the two features described in the previous section: Firstly, a uniform correction in the sense of \eqref{equ:UniformError} is impossible due to different kernel spaces. However, a learned correction in the mean \eqref{equ:StatisticalCorrection} is possible on some data $\{\trni{x}{i}, i = 1,\ldots,N\}$ consisting of piecewise constant functions: On these samples the two operators $\ApproxOp$ and $\TrueOp$ already coincide everywhere 
except near jumps, where a weighted average can be employed to correct the approximation error.



\subsection{Solving the variational problem}
We now aim to solve an inverse problem given the corrected model $\CorrectedOp$ by solving the associated variational problem  \eqref{eqn:varProbIntro}. 
In this context it is natural to require that the solutions of the two minimisation problems, involving the operator correction $\CorrectedOp$ and $\TrueOp$, are close, that is 
\begin{align} \label{equ:variationalSolutionEquality}
\argmin_{x\in X} \frac{1}{2} \|\CorrectedOp(x) - y\|_Y^2 + \lambda \RegularisationOp(x)
\approx
\argmin_{x\in X}  \frac{1}{2} \|\TrueOp x - y\|_Y^2 + \lambda \RegularisationOp(x).
\end{align}
Note that this formulation is different to the approximation error method \eqref{eqn:varProbBAE}, where the data fidelity term is given by $\|L_{\varepsilon} (\ApproxOp x - y + \eta_{\varepsilon})\|_Y^2$.
Solutions to \eqref{eqn:varProbIntro} are then usually computed by an iterative algorithm.
Here we consider first order methods to draw connections to learned iterative schemes \cite{Adler2017,Adler2018,hammernik2018learning}. In particular, we consider a classic gradient descent scheme, assuming differentiable $\RegularisationOp$. Then, given an 
initial guess $x_0$, we can compute a solution by the following iterative process
\begin{equation}\label{equ:proximalGrad}
x_{k+1} = x_k - \gamma_k \nabla_x\left(\frac{1}{2}\|Ax_k-y\|_X^2 +\lambda \RegularisationOp(x_k)\right),
\end{equation}
with appropriately chosen step size $\gamma_k>0$. When using \eqref{equ:proximalGrad} for the corrected operator it seems natural to ask for
a \emph{gradient consistency} \revision{of the approximate gradient}
\begin{align}\label{equ:gradConsistencyCondition}
\nabla_x \|\CorrectedOp(x) - y\|_X^2 \approx \nabla_x \|\TrueOp x - y\|_X^2
\end{align}
and hence we can identify
\begin{align}
\label{equ:GradientMisalignement}
\sum_{i=1}^N \left\| \nabla_x \left\| \CorrectedOp(x^i) - y^i\right\|_X^2 - \nabla_x \left\|\TrueOp x^i - y^i\right\|_X^2 \right\|
\end{align}
as another relevant measure of quality for model corrections within the  variational framework, if gradient schemes are used to solve \eqref{eqn:varProbIntro}. In the following we will discuss possibilities to 
obtain a correction, such that we can guarantee a closeness of solutions in the sense of \eqref{equ:variationalSolutionEquality}.


\section{Forward model correction}\label{sec:forwardCorrection}

We will now present the possibility to correct the forward model only and discuss resulting shortcomings of this approach. More precisely, in a forward model 
correction, the approximate operator $\ApproxOp:X\to Y$ is corrected using a 
neural network $\ForwardCor: Y \to Y$ that is trained to remove artefacts in
data space for a given training set. This leads to a corrected operator of the form
$\CorrectedOp = F_\Theta \circ \ApproxOp$.

\subsection{The adjoint problem}
To solve the minimisation problem \eqref{eqn:varProbIntro} with the learned 
forward operator with an iterative scheme such as \eqref{equ:proximalGrad}, we need to compute the gradient of the data fidelity. We recall 
that the corrected operator $\CorrectedOp = \ForwardCor \circ \ApproxOp$ where the correction $\ForwardCor$ is given by a nonlinear neural network. Following the chain rule we obtain the following gradient 
\begin{equation}\label{eqn:fidelityTerm}
\frac{1}{2}\nabla_x \|\CorrectedOp(x) - y\|_2^2 = \ApproxOp^*  \left[D\ForwardCor(\ApproxOp x)\right]^* \left(\ForwardCor(\ApproxOp x)-y\right).
\end{equation}
Here, we denote by $D\ForwardCor(y)$ the Fr{\'{e}}chet derivative of $F_\Theta$ at $y$, which is a linear operator $Y \to Y$. Whereas the gradient for the correct data fidelity term is simply given by 
\[
\frac{1}{2}\nabla_x \|A x - y\|_Y^2 = {A}^* ({A}x-y).
\]
That means, to satisfy the gradient consistency condition \eqref{equ:gradConsistencyCondition}, we would need
\begin{equation}\label{equ:gradConsistencyApprox}
\ApproxOp^* \left[D\ForwardCor(\ApproxOp x)\right]^* \left(\ForwardCor (\ApproxOp x)-y\right) \approx A^*(Ax-y).
\end{equation}
On the other hand, if we train the forward model 
correction, only requiring consistency in data space by minimising \eqref{equ:StatisticalCorrection}, we will only ensure consistency of the
residuals $\ForwardCor(\ApproxOp x)-y\approx Ax-y$, but not full gradient 
consistency as in \eqref{equ:gradConsistencyCondition}. In order to enforce gradient consistency we need to
control the derivative of the network $D\ForwardCor(\ApproxOp x)$
and consequently also need to take the adjoint into consideration when training the
forward correction. This could be done by adding an additional penalty term to \eqref{equ:StatisticalCorrection} that penalises the network for
exhibiting an adjoint different from $A^*$. For that purpose, let us examine the 
adjoint of the linearisation of the correction operator $\CorrectedOp$ around a point $x$
\begin{align*}
\left( D \CorrectedOp (x) \right)^*[y]
= \ApproxOp^* \left( D \ForwardCor(\ApproxOp x) \right)^*[y].
\end{align*}
With this we can consider the following additional penalty term in the training 
\begin{equation}\label{eqn:adjointPenalty}
\left\|\left(A^*-\ApproxOp^*\circ \left[D\ForwardCor(\ApproxOp x) \right]^*\right )(r)\right \|_X,\quad \textnormal{where}\ r=\ForwardCor( \ApproxOp x)-y
\end{equation}
and choose $r$ to be the residual in data 
space $\ForwardCor (\ApproxOp x) - y$ that arises when minimising the data fidelity term as in \eqref{eqn:fidelityTerm}.

However this solution comes with its own drawback. As we can see in
\eqref{eqn:fidelityTerm}, \revision{the range of the corrected fidelity term's gradient \eqref{eqn:fidelityTerm} is limited  by the range of the approximate 
	adjoint, $\range{\ApproxOp^*}$. Thus, we identify the key difficulty here in the differences of the range of the accurate and
	the approximate adjoints rather than the differences in the forward operators themselves, which links back to the discussion 
	in \ref{sec:learnedDataCor}.}

Indeed, a correction of the forward operator via composition with a parametrised model $F_\Theta$ in measurement space is not able to yield gradients close to the gradients of the accurate data term if $\range{\ApproxOp^*}$ and $\range{A^*}$ are too different. \revision{This problem is exacerbated if the dimensions of these two spaces differ and we can not expect to find a correction} that satisfies the gradient 
consistency \eqref{equ:gradConsistencyApprox} and, related to Remark \ref{remark:ApOpRangeDomain}, even suitable projections in $\ApproxOp$ would not be sufficient to compensate for this.
This observation can be made precise in the following theorem. 


\begin{theorem}[Unlearnability of a gradient consistent forward model correction]
	\label{thm:UnlearnableAdjoint}
	Let $A$ and $\ApproxOp$ be compact linear operators from $X$ to $Y$ and given the solutions
	\begin{align}
	\label{eqn:varProbsTheo1} &\hat{x} \in \argmin_x \frac{1}{2}\|Ax-y\|_Y^2 \\ 
	\label{eqn:varProbsTheo2}  &\hat{x} _\Theta \textnormal{ critical point of } \frac{1}{2} \|\CorrectedOp(x)-y\|_Y^2.
	\end{align}
	If $\iter{\widetilde{x}}{0} \in  \range{\ApproxOp^*}$ and $\hat{x} \notin \overline{\range{\ApproxOp^*}}$, 
	then a gradient-descent algorithm for the functional in \eqref{eqn:varProbsTheo2}, initialised with $\iter{\widetilde{x}}{0}$, yields a solution
	such that $\hat{x}_\Theta \neq \hat{x}$ for any $\hat{x}$ solving \eqref{eqn:varProbsTheo1}.
\end{theorem}
\begin{proof}
	This follows directly from the update equations for solving \eqref{eqn:varProbsTheo2} by
	\[
	\iter{\widetilde{x}}{k+1} = \iter{\widetilde{x}}{k} - \lambda_k \Delta \iter{\widetilde{x}}{k} 
	\]
	with 
	\begin{equation}\label{equ:approximateUpdateEqu}
	\Delta \iter{\widetilde{x}}{k} :=\frac{1}{2} \nabla_{\iter{\widetilde{x}}{k}}\|\CorrectedOp( \iter{\widetilde{x}}{k})-y\|_Y^2=\ApproxOp^* \left[D\ForwardCor(\ApproxOp\iter{\widetilde{x}}{k})\right]^* \left(\ForwardCor( \ApproxOp\iter{\widetilde{x}}{k})-y\right).
	\end{equation}
	If $\iter{\widetilde{x}}{0}\in \range{\ApproxOp^*}$ then $\Delta \iter{\widetilde{x}}{0} \in  \range{\ApproxOp^*}$, and hence $\iter{\widetilde{x}}{1}\in \range{\ApproxOp^*}$. By induction this is true for all $k>0$, i.e. $\iter{\widetilde{x}}{k}\in \range{\ApproxOp^*}, ~\forall k$ and thus any limit point $\hat{x} _\Theta \in \overline{\range{\ApproxOp^*}}$ lies in the closure of the range of $\ApproxOp^*$.
	Since $\hat{x} \notin \overline{\range{\ApproxOp^*}}$ it follows that $\hat{x} \neq \hat{x}_\Theta$ for any limit point of a gradient-descent algorithm for solving \eqref{eqn:varProbsTheo2}.
\end{proof}


Thus, a correction of the forward model by requiring only consistency in data 
space does not in fact ensure consistency of the data term, when solving a 
variational problem. Additionally, according to Theorem \ref{thm:UnlearnableAdjoint} even by including an additional penalty term in the form of \eqref{eqn:adjointPenalty} does not solve this problem.



\subsubsection{Illustration with the toy case}
Going back to the toy case from Section \ref{Sec:DownsamplingToyCase}, where we 
considered a downsampling operation. The approximate operator was chosen such 
that the null space is spanned by the unit vectors with \revision{even} index. The range of 
the adjoint can then be characterised by the identity $\range{\ApproxOp^*} = (\kernel{\ApproxOp})^\perp$ and hence we 
have $\range{\ApproxOp^*}=\{\mathbf{e}_j\ | \ 0\leq j\leq n, \ j \  \revision{\textit{odd}} \}$. It is now clear, that we cannot compute any solution $x^*\notin \range{\ApproxOp^*}$ by the updates in \eqref{equ:approximateUpdateEqu}, if we initialise them with $\iter{\widetilde{x}}{0}\in \range{\ApproxOp^*}$, since all updates are restricted to 
the range of the adjoint of the approximate operator. This problem is illustrated in  
Figure \ref{fig:toyCase_mappings}, where we consider an imaging problem for illustrative purposes and $x$ is vectorised before the operators in \eqref{equ:downsamplingOps} are applied. 
Whereas the difference in the forward operator is minimal for this example, the range of the approximate adjoint makes it impossible to recover the phantom without further adjustments after application of the adjoint, which will be addressed in the next section.

\begin{figure}[ht!]
	\centering
	\begin{subfigure}{300pt}
		\begin{picture}(400,100)
		\put(0,0){\includegraphics[width=110pt]{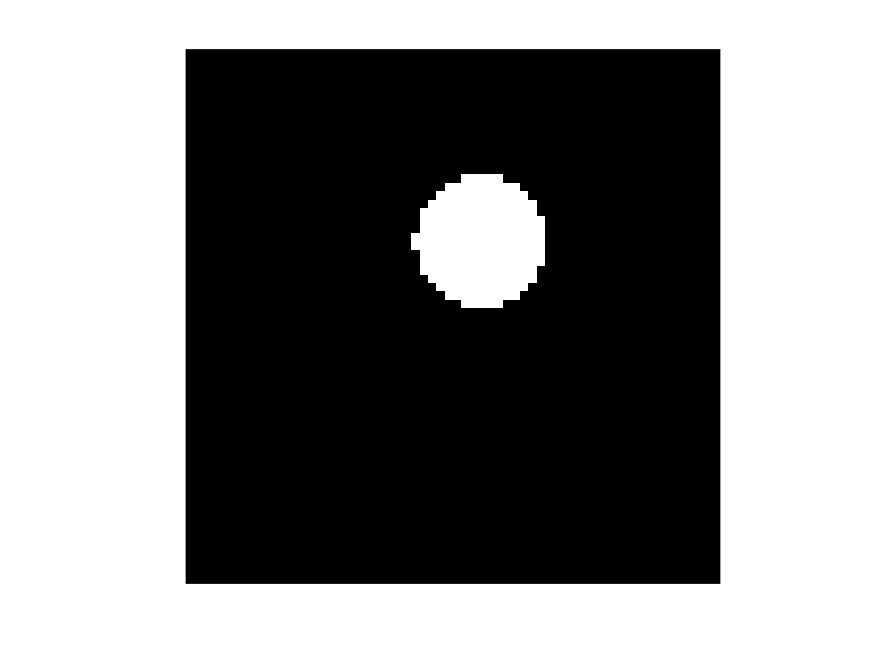}}
		\put(100,0){\includegraphics[width=110pt]{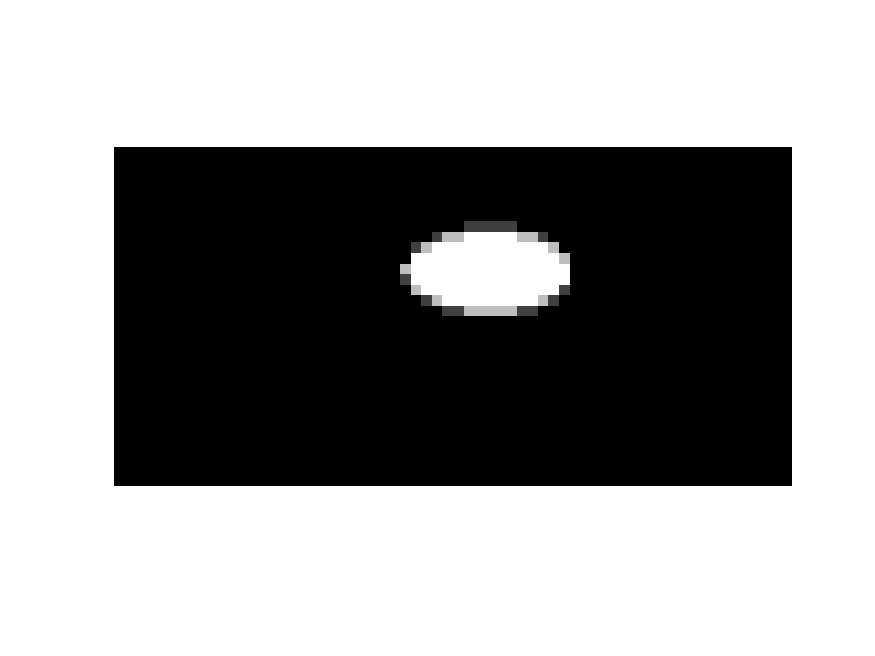}}
		\put(200,0){\includegraphics[width=110pt]{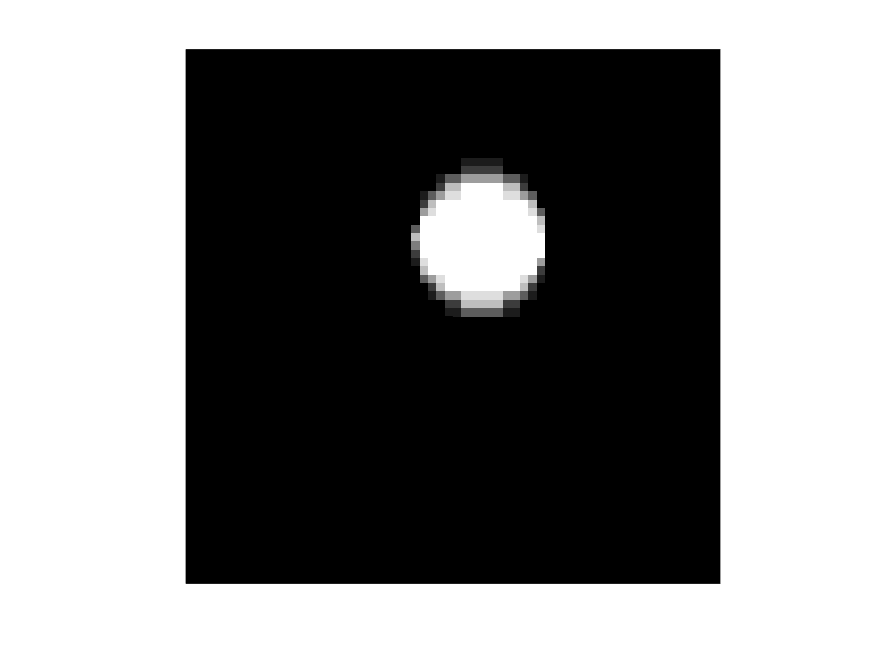}}
		\put(30,80){Phantom $x$}
		\put(130,80){Forward $Ax$}
		\put(227,80){Adjoint $A^*Ax$}
		
		\end{picture}
		\caption{\revision{Application of the accurate forward operator and its adjoint}}
	\end{subfigure}
	\begin{subfigure}{300pt}
		
		\begin{picture}(400,100)
		\put(0,0){\includegraphics[width=110pt]{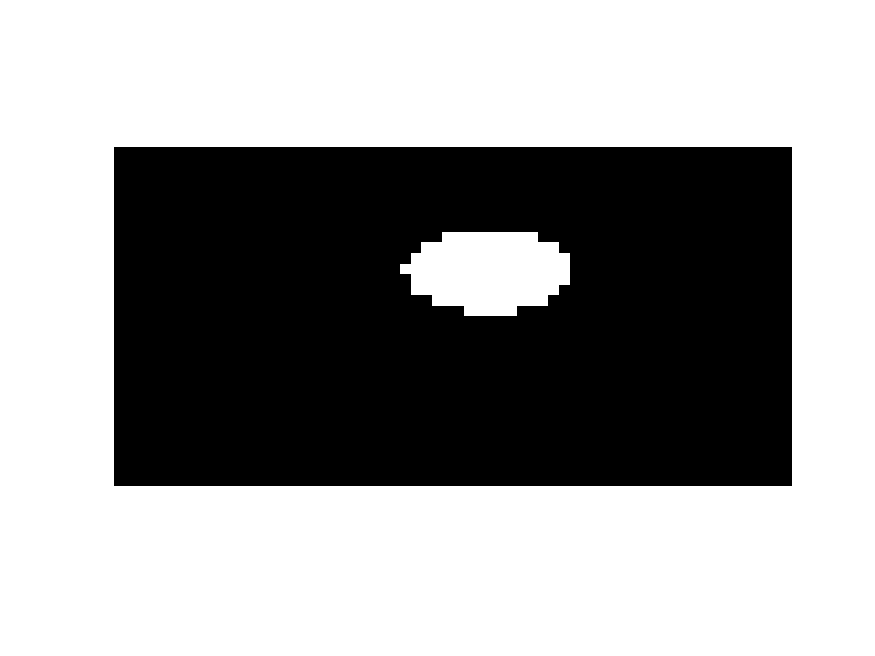}}
		\put(100,0){\includegraphics[width=110pt]{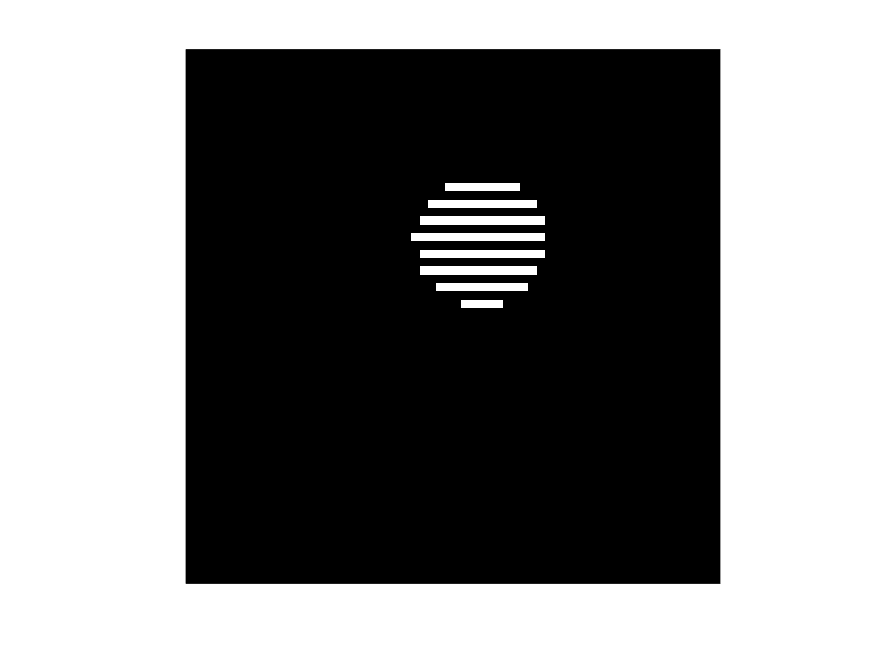}}
		\put(200,0){\includegraphics[width=110pt]{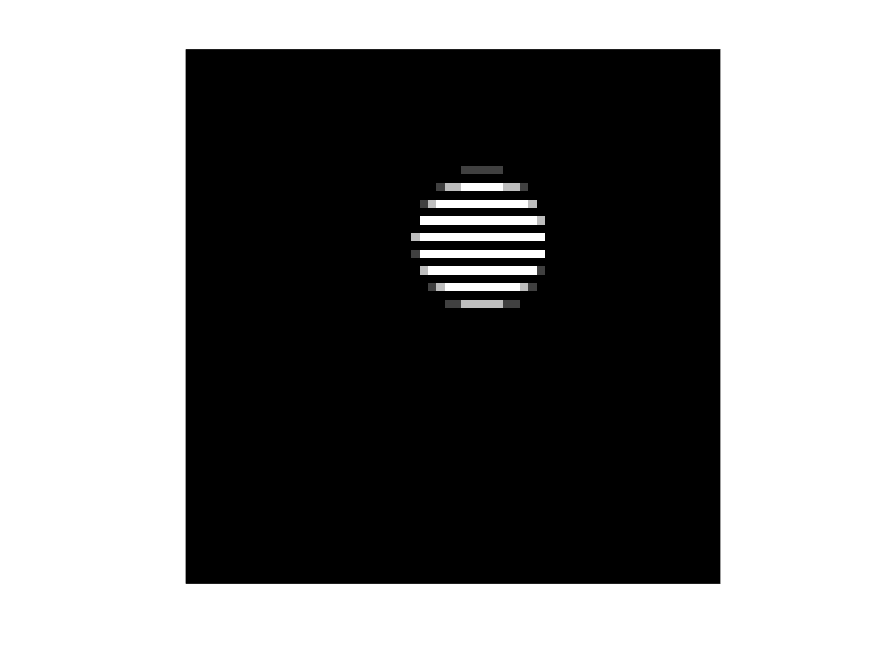}}
		\put(30,80){Forward $\ApproxOp x$}
		\put(127,80){Adjoint $\ApproxOp^*\ApproxOp x$}
		\put(227,80){Adjoint $\ApproxOp^*Ax$}
		\end{picture}
		\caption{\revision{Application of the approximate forward operator and its adjoint}}
	\end{subfigure}
	
	\caption{\label{fig:toyCase_mappings} Illustration of mapping properties for the toy case. As we can see, the range of the adjoint and approximate adjoint are essentially different. Even if the approximate adjoint $\ApproxOp^*$ is applied to the ideal data $Ax$ (bottom right), representing a perfect fit of the forward model, the range of the approximate adjoint $\range{\ApproxOp^*}$ makes it impossible to compute a consistent gradient in \eqref{equ:gradConsistencyCondition} without further modifications.}
\end{figure}
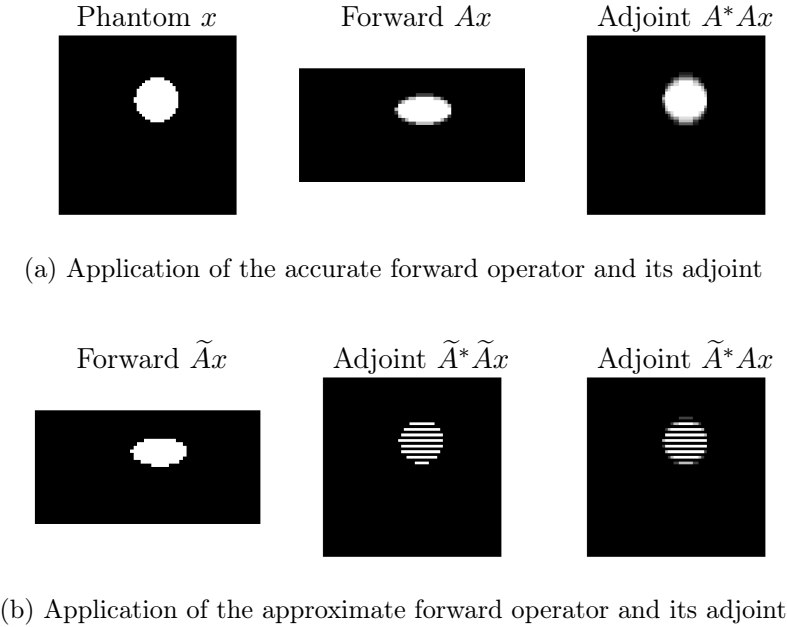

\section{Forward-Adjoint Correction}\label{sect:ForwardBackwardCorr}
As is evident from the last section, a forward model correction that is computed to minimise \eqref{equ:StatisticalCorrection} in data space alone is not sufficient to compute the actual reconstruction in a variational framework. We additionally require consistency in the gradients of the data fidelity term \eqref{equ:GradientMisalignement} which in turn boils down to a condition for a correction on the adjoint of the corrected forward operator in image space, motivated by \eqref{eqn:adjointPenalty}.
We will refer to such a correction in data and image space
as a \emph{forward-adjoint correction}, as we will learn a correction of the 
forward operator, as well as a correction of the adjoint (backward).


\subsection{Obtaining a Forward-Adjoint Correction}
The goal is now to obtain a gradient consistent model correction. To achieve this 
we propose to learn two networks. That is, we learn a network $\ForwardCor$ that 
corrects the forward model and another network $\AdjointCor$ that corrects the 
adjoint, such that we have
\begin{align*}
\CorrectedOp := \ForwardCor \circ \ApproxOp, \quad \CorrectedAd := \AdjointCor \circ \ApproxOp^*
\end{align*}
These corrections are obtained as follows. Given a set of training samples $(\trni{x}{i}, A\trni{x}{i})$, we train the forward correction $\ForwardCor$ acting in measurement space $Y$ with the loss
\begin{align}\label{equ:forwardLoss_FBC}
\min_\Theta \sum_i \| \ForwardCor(\ApproxOp \trni{x}{i}) - A \trni{x}{i} \|_Y.
\end{align}
In an analogous way, we correct the adjoint with the network $\AdjointCor$
acting on image space $X$ with the loss
\begin{align}\label{equ:adjointLoss_FBC}
\min_\phi \sum_i \| \AdjointCor(\ApproxOp^* \trni{r}{i}) -  A^* \trni{r}{i} \|_X.
\end{align}
Here, we can choose the direction $\trni{r}{i} = \ForwardCor( \ApproxOp \trni{x}{i})- \trni{y}{i} $ as 
in \eqref{eqn:adjointPenalty} for the adjoint loss.
This ensures that the adjoint correction is in fact trained in directions relevant when solving the variational problem.

At evaluation time, the corrected operators can then be used to compute 
approximate gradients of the data fidelity term $\|Ax-y\|_Y^2$. The
gradient then takes the form
\begin{align}
\label{eqn:GradientFit}
A^*(Ax-y) & \approx \left(\AdjointCor\circ \ApproxOp^*\right) \left( \ForwardCor(\ApproxOp x) - y \right).
\end{align}

Let us note that the separate correction of the adjoint and the forward operator comes with a change of philosophy compared to existing methods for forward operator correction as presented in Section \ref{sect:AEM}. Instead of trying to fit a single 
corrected operator $\CorrectedOp$ that is already parametrised according to its use within the data fidelity term of a variational problem, we fit a nonlinear corrected operator $\CorrectedOp$ whose use within the variational problem requires to fit the \textit{gradient} of the data term directly. 
This gradient fit takes the form as in \eqref{eqn:GradientFit}.
\revision{
	We use the gradient of the data fidelity term to directly obtain the gradient of the variational functional for our corrected operator, allowing us to perform minimisation techniques like gradient descent. We take the obtained critical point of these dynamics as the reconstruction. Note that the approximate gradient can not be associated to a variational functional for the forward-adjoint method anymore. Instead, the gradient is parametrised directly, without parametrising the variational functional first.
}

\begin{remark}
	We note, that such a separate correction in image and data space can be related to learned primal dual (LPD) methods \cite{Adler2018}, where the correction is performed implicitly as described in Section \ref{sec:learnModelCor}. This explains in part why LPD approaches might be especially suitable for applications with an imperfectly known operator, see also \cite{vishnevskiy2019deep}.
\end{remark}

In the following section we will discuss how these dynamics relate to the 
original variational problem and we will see that they can in fact take us close
to the original reconstruction if both the forward and adjoint are fit 
sufficiently well.

\subsection{Convergence Analysis}
\label{sec:ConvergenceAnalysis}

The purpose of this section is to show that sufficiently small training losses can ensure that \revision{gradient descent over \eqref{eqn:varProbIntro}} converges to a neighbourhood of the reconstruction $\hat{x}$, obtained with the accurate operator $\TrueOp$. \revisionTwo{The section relates to the forward-adjoint correction \eqref{eqn:GradientFit} and uses the notation of this approach.}
In the case of forward-adjoint correction, these loss functions are given by
\begin{align}
\| \TrueOp x - \CorrectedOp(x) \|_Y \text{ and } \| \left( \TrueOp^* - \CorrectedAd \right)(\CorrectedOp(x) - y) \|_X .
\end{align}.

Let us now consider for any $y \in Y$ the two functionals 
\begin{align*}
\label{eqn:DefinitionL}
\mathcal{L}(x) &:= \frac{1}{2} \| \TrueOp x - y \|_Y^2 + \lambda R(x),
\\
\mathcal{L}_\Theta(x) &:= \frac{1}{2} \| \CorrectedOp(x) - y \|_Y^2 + \lambda R(x)
\end{align*}
associated with the variational problem for the reconstruction $x$ from the measurement $y$. We will show connections between the reconstruction  $\hat{x} := \argmin_x \mathcal{L}(x)$ using the accurate operator $\TrueOp$ and the solutions $\hat{x}_\Theta \in \argmin_x \mathcal{L}_\Theta(x)$ obtained with our corrected operator $\CorrectedOp$.

When considering the gradient descent dynamics over $\mathcal{L}_\Theta$, we do not refer to the actual gradient over $\mathcal{L}_\Theta$ but instead consider the direct fit to the gradient of the form $\CorrectedAd(\CorrectedOp(x) - y) + \lambda \nabla R (x)$ as discussed in the last section. In a slight abuse of notation we will nevertheless denote this gradient as $\dnabla \mathcal{L}_\Theta := \CorrectedAd(\CorrectedOp(x) - y) + \lambda \nabla R (x)$ to keep the notation easy to read in the remainder of this section. If $R$ is merely sub-differentiable, then $\nabla R (x)$ denotes an element in the subgradient of $R$.


For the remainder of this chapter, we  make the following assumption on the regularisation functional $R$. 

\begin{assumption}[Strong Convexity]
	\label{ass:StrongConvexity}
	We assume that the regularisation functional $R$ is strongly convex and denote the strong convexity constant by $m$.
\end{assumption}

\begin{remark}
	Assumption \ref{ass:StrongConvexity} in particular holds for $R$ being the Tikhonov regularisation functional $R(x) = \| x \|_X^2$ and for the pseudo-Huber loss 
	$R(x) = \int_{[0,1]^2} \delta \left[ \sqrt{ 1 + \frac{1}{\delta^2} \|\nabla_t x(t)\|^2} - 1 \right] $ for a bounded function $x: [0,1]^2 \mapsto \mathbb{R}$ and $\delta > 0$ which we use in the experimental section. \revisionTwo{For operators $A$ with bounded inverse it is sufficient for the regularisation functional to be convex to ensure strong convexity of the resulting variational functional $\mathcal{L}$. In this case, strong convexity of the regularisation functional is not required.}
\end{remark}

This allows us to use the following two fundamental lemmas on the behaviour of $\mathcal{L}$ near the minimum of the variational functional. \revision{As a direct consequence of \ref{ass:StrongConvexity} and the convexity of the data term for linear forward operators we will from now on assume $\mathcal L$ to be strongly convex.}

\begin{lemma}[Proximity to minimiser]
	\label{lem:LowEnergyCloseness}
	Let $\mathcal L$ be strongly convex. Then for every $\epsilon$ there is a $\delta >0 $ such that for any $y$ and $x$ with
	\begin{equation}
	\label{eqn:VaritionalAssumptionLoss}
	\mathcal{L}(x) - \mathcal{L}(\hat{x}) \leq \delta \implies \|x -\hat{x}\|_X \leq \epsilon,
	\end{equation}
	where $\hat{x} := \argmin_x \mathcal{L}(x)$.
\end{lemma}
\begin{proof}
	By the definition of strong convexity we have
	\begin{align*}
	\mathcal{L}(x) \geq \mathcal{L}(\hat{x}) + \langle s_{\hat{x}}, x - \hat{x} \rangle_X + \frac{m}{2} \| x - \hat{x}\|_X^2,
	\end{align*}
	where $s_{\hat{x}} \in \partial \mathcal{L}(\hat{x})$ 
	is in the subdifferential of $\mathcal{L}$ at $\hat{x}$. Using $0 \in \partial \mathcal{L}(\hat{x})$ yields
	\begin{align*}
	\delta \geq  \mathcal{L}(x) - \mathcal{L}(\hat{x}) \geq \frac{m}{2} \| x - \hat{x}\|_X^2
	\end{align*}
	which proves the claim by setting $\delta = \frac{2 \epsilon}{m}$.
\end{proof}

\begin{lemma}[Lower Gradient Norm Bound]
	\label{lem:LowerBoundGradientNorm}
	Let $\mathcal L$ be strongly convex. For every $\epsilon$ there is a $\delta >0 $ such that for any $y$ and $x$ with
	\begin{align}
	\|x -\hat{x}\|_X > \epsilon \implies \forall s \in \partial \mathcal{L}(x) : \| s \|_X > \delta,
	\label{eqn:VaritionalAssumptionGradient}
	\end{align}
	where $\partial \mathcal{L}(x)$ denotes the subdifferential of $\mathcal{L}$ at $x$ and $\hat{x} := \argmin_x \mathcal{L}(x)$.
	\label{lemma:AlignementBound}
\end{lemma}

\begin{proof}
	By the definition of strong convexity
	\begin{align*}
	\mathcal{L}(\hat{x}) \geq \mathcal{L}(x) + \langle s_x, \hat{x} - x \rangle_X + \frac{m}{2} \| x - \hat{x}\|_X^2,
	\end{align*}
	where again $s_x$ denotes an element in the subdifferential of $\mathcal{L}$ around $x$. Then by Cauchy-Schwarz
	\begin{align*}
	\mathcal{L}(\hat{x}) - \mathcal{L}(x) - \frac{m}{2} \| x - \hat{x}\|_X^2 \geq - \|s_x\|_X \|\hat{x} - x\|_X.
	\end{align*}
	Using $\mathcal{L}(\hat{x}) - \mathcal{L}(x) < 0$ by assumption shows
	\begin{align*}
	\frac{m}{2} \| x - \hat{x}\|_X^2 \leq \|s_x\|_X \|\hat{x} - x\|_X,
	\end{align*}
	and hence $\|s_x\|_X \geq \frac{m}{2} \| x - \hat{x}\|_X$, which proves the result.
\end{proof}

\begin{remark}
	The assumption of strong convexity is used in the following results via the lemmas \ref{lem:LowEnergyCloseness} and \ref{lem:LowerBoundGradientNorm} only. While it is a sufficient condition for these to hold, it is not necessary. In particular, if the variational functional is not strongly convex but such that \ref{lem:LowEnergyCloseness} and \ref{lem:LowerBoundGradientNorm} hold true, the following results still apply.
\end{remark}

We now turn to show that a minimiser $\hat{x}_\Theta$ of the approximate functional can in fact be computed with a gradient descent scheme and that this such a minimiser is in fact to the accurate reconstruction $\hat{x}$. We begin by extending Lemma~\ref{lemma:AlignementBound} 
to include the regularisation term. For this purpose, we consider the alignment of the variational gradients including the regularisation term
\begin{align}
\label{eqn:AligenementVariational}
\cos \Phi_v(x) := \frac{\langle \nabla \mathcal{L}(x), \dnabla  \mathcal{L}_\Theta(x) \rangle}{\| \nabla \mathcal{L}(x)\|^2}.
\end{align}
We show how the alignment can be used as key quantity to guarantee convergence of the approximate dynamics to a a neighbourhood of the accurate solution. We remark again the abuse of notation $\dnabla  \mathcal{L}_\Theta(x) := \CorrectedAd(\CorrectedOp(x) - y) + \lambda \nabla R (x)$.



\begin{proposition}[Convergence under alignment constraints]
	\label{prop:convergenceAlignement}
	Assume that outside a neighbourhood $U$ of the minimiser $\hat{x}$ of the exact functional $\mathcal{L}$ we have
	\begin{align*}
	\cos\Phi(x) > \delta_1 > 0,
	\end{align*}
	for some $\delta_1 >0$. Then eventually the gradient descent dynamics over $\mathcal{L}_\Theta$ will reach the neighbourhood $U$.
\end{proposition}

\begin{proof}
	Denote by $x_\Theta(t)$ the trajectory of the reconstruction under the gradient flow
	\begin{align*}
	\partial_t x_\Theta(t) = - \dnabla \mathcal{L}_\Theta(x_\Theta(t)).
	\end{align*}
	Consider the evaluation of the variational loss $\mathcal{L}$ that invokes the correct forward operator $\TrueOp$. Using the bound of the alignment as in Lemma~\ref{lemma:AlignementBoundVariational}, we can bound
	\begin{align*}
	\partial_t \mathcal{L}(x_\Theta(t)) &=   \langle \nabla \mathcal{L}(x_\Theta(t)), \partial_t x_\Theta(t) \rangle_X = - \langle  \dnabla \mathcal{L}(x_\Theta(t)) , \nabla \mathcal{L}_\Theta (x_\Theta(t)) \rangle_X
	\\
	&\leq - \delta_1 \cdot \|\nabla \mathcal{L}(x)\|_X^2 .
	\end{align*}
	As long as $\Theta(t)$ has not reached the neighbourhood $U$, by \eqref{eqn:VaritionalAssumptionGradient}, we have $\|\nabla \mathcal{L}(x)\|_X > \delta_2$ for some $\delta_2$ and hence
	\begin{align*}
	\partial_t \mathcal{L}(x_\Theta(t)) \leq - \delta_1 \cdot \|\nabla \mathcal{L}(x)\|_X^2 \leq - \frac{1}{2} \delta_1 \delta_2 =: -c <0.
	\end{align*}
	The gradient flow dynamics induced by $\dnabla \mathcal{L}_\Theta$ hence induce a decrease of $\mathcal{L}$ at a rate that is globally bounded by $c$ outside neighbourhood $U$ around $\hat{x}$, concluding the proof by lemma \ref{lem:LowEnergyCloseness}.
\end{proof}

We have shown that even though the corrected operator $\CorrectedOp$ is potentially nonlinear, the gradient dynamics induced by $\dnabla \mathcal{L}_\Theta$ can in fact minimise the variational problem with the accurate operator $\TrueOp$, effectively minimising the associated variational functional $\mathcal{L}$ and leading us close to the accurate solution $\hat{x}$. \revision{The proposition is based on an assumption about the alignment $\cos \Theta$. We will directly track this quantity in our experimental section, making sure the convergence results can be applied to our experimental findings. The training loss, however, is not based on the alignment directly, but rather minimises a combination of forward and adjoint loss. We have in fact found that this combination of loss functionals is both more interpretable and more stable than directly minimising alignment. The following lemma and theorem show that these loss functions in fact minimise a lower bound on the alignment and hence a sufficiently well-trained correction can also be guaranteed to yield results close to the minimiser $\hat{x}$ of the variational functional involving the exact operator $\TrueOp$. In this context, a well-trained correction is such that it achieves sufficiently low training errors.}

\begin{lemma}[Complete gradient alignment bound]
	\label{lemma:AlignementBoundVariational}
	Let $\mathcal{L}$ and $\mathcal{L}_\Theta$ be defined as above. We have the lower bound
	\begin{align*}
	\cos\Phi_v \geq 1 - \frac{\|\TrueOp\|_{X \to Y} \|(\TrueOp - \CorrectedOp)(x) \|_Y + \| (\TrueOp^* - \CorrectedAd) (\CorrectedOp(x) - y) \|_X}{\| \nabla \mathcal{L}(x)\|_X},
	\end{align*}
	where $\cos\Phi_v$ is defined as in \eqref{eqn:AligenementVariational}.
\end{lemma}
\begin{proof}
	A straightforward calculation shows
	\begin{align*}
	&\frac{\langle \nabla \mathcal{L}(x), \dnabla  \mathcal{L}_\Theta(x) \rangle_X}{\| \nabla \mathcal{L}(x)\|_X^2}
	\\
	= \ &\frac{\langle \nabla \mathcal{L}(x), \nabla  \mathcal{L}(x) \rangle_X}{\| \nabla \mathcal{L}(x)\|_X^2} + \frac{\langle \dnabla \mathcal{L}_\Theta(x) - \nabla \mathcal{L}(x), \nabla  \mathcal{L}(x) \rangle_X}{\| \nabla \mathcal{L}(x)\|_X^2}
	\\
	\geq \ &1 - \frac{\|\dnabla \mathcal{L}_\Theta(x) - \nabla \mathcal{L}(x)\|_X}{\| \nabla \mathcal{L}(x)\|_X}
	\end{align*}
	The result follows by using the bound
	\begin{align*}
	&\| \TrueOp^* (\TrueOp x- y) - \CorrectedAd (\CorrectedOp( x) - y) \|_X
	\\
	\leq \ &\|\TrueOp\|_{X \to Y} \|(\TrueOp - \CorrectedOp)(x) \|_Y + \| (\TrueOp^* - \CorrectedAd) (\CorrectedOp(x) - y) \|_X,
	\end{align*}
	which itself emerges directly from the triangular inequality applied to the identity
	\begin{align*}
	\TrueOp^* (\TrueOp x- y) - \CorrectedAd (\CorrectedOp(x) - y)
	=
	\TrueOp^*(\TrueOp - \CorrectedOp)(x) + (\TrueOp^* - \CorrectedAd) (\CorrectedOp(x) - y).
	\end{align*}
\end{proof}

\begin{theorem}[Convergence to a neighbourhood of $\hat{x}$]\label{theo:conv2Neighbourhood}
	\label{thm:GradientConvergence}
	Let $\epsilon>0$ and pick $\delta$ as in \eqref{eqn:VaritionalAssumptionGradient}. Assume both adjoint and forward operator are fit up to a $\delta/4$-margin, i.e.
	\begin{align}\label{eqn:conditionTheo410}
	\|\TrueOp\|_{X \to Y} \| (\TrueOp - \CorrectedOp)(x_n) \|_Y < \delta/4, \quad \|(\TrueOp^* - \CorrectedAd) (\CorrectedOp(x_n) - y) \|_X < \delta/4
	\end{align}
	for all $y$ and $x_n$ obtained during gradient descent over $\mathcal{L}_\Theta$. 
	Then eventually the gradient descent dynamics over $\mathcal{L}_\Theta$ will reach an $\epsilon$ neighbourhood of the accurate solution $\hat{x}$.
\end{theorem}

\begin{proof}
	We apply \ref{prop:convergenceAlignement}, with the neighbourhood $U$ chosen as the $\epsilon$ ball around $\hat{x}$. Using  Lemma~\ref{lemma:AlignementBoundVariational}, we can bound
	\begin{align*}
	\cos\Phi \geq 1 - \frac{\|\TrueOp\|_{X \to Y} \|(\TrueOp - \CorrectedOp)(x) \|_Y + \| (\TrueOp^* - \CorrectedAd) (\CorrectedOp(x) - y) \|_X}{\| \nabla \mathcal{L}(x)\|_X} \geq 1- \frac{\delta/4 + \delta/4}{\| \nabla \mathcal{L}(x)\|_X}
	\end{align*}
	As long as $\|x_\Theta(t) - \hat{x} \|_X \geq \epsilon$, by \eqref{eqn:VaritionalAssumptionGradient}, we have $\|\nabla \mathcal{L}(x)\|_X > \delta$ and hence
	\begin{align*}
	\cos\Phi \geq 1 - \frac{\delta/2}{\delta} >0
	\end{align*}
	We can hence apply \ref{prop:convergenceAlignement} to conclude the proof.
\end{proof}

Overall, we have thus shown that a sufficiently well-trained non-linear corrected operator $\CorrectedOp$ induces gradient dynamics $\dnabla \mathcal{L}_\Theta$ that lead close to the accurate solution $\hat{x}$.

We note that the main assumption in Theorem \ref{theo:conv2Neighbourhood} is \revision{that the} learned operator $\CorrectedOp$ has to be sufficiently close to the accurate operator $\TrueOp$ \textit{throughout} the minimisation trajectory, in the sense of \eqref{eqn:conditionTheo410}.
While this corresponds directly to the quantities of the loss functions that the approximations $\CorrectedOp$ and $\CorrectedAd$ were trained on, it includes any $x_n$ occurring during the gradient descent dynamics. Thus, we will discuss the concept of adding exactly these samples $x_n$ to the training set in the next chapter, effectively making our training loss function minimise exactly the relevant quantities $ \| (\TrueOp - \CorrectedOp)(x_n) \|_Y$ and $\|(\TrueOp^* - \CorrectedAd) (\CorrectedOp(x_n) - y) \|_X$.

\begin{remark}
	\revision{
		The above Theorem \ref{thm:GradientConvergence} makes use of both proximity of the forward operator as well as of the adjoints. While this is necessary to guarantee convergence of the gradient descent dynamics to a neighbourhood of the accurate solution, it is not strictly necessary to guarantee proximity of the minimisers of $\mathcal{L}_\Theta$ and of $\mathcal{L}$. In fact, in Appendix \ref{sec:appxTheo} we show that under certain assumptions a good forward approximation quality is sufficient to ensure closeness of minimisers, without considering a specific optimisation scheme. 
		While this result is interesting from a theoretical viewpoint, Theorem \ref{thm:GradientConvergence} is essential for supporting and explaining the experimental results in this study.
	}
\end{remark}




\section{Computational considerations}
\label{sec:ComputationalConsiderations}
In the following we will first address some details on the training procedures and then continue to present the design of experiments to evaluate performance of the discussed approaches. In particular, as we mentioned above, in order to ensure the convergence in Theorem \ref{theo:conv2Neighbourhood}, we need to make sure that the forward fit as well as the backward fit in \eqref{eqn:conditionTheo410} are satisfied throughout the minimisation process, which makes a special recursive training of the corrections necessary.

\subsection{Recursive training}
\label{sec:RecursiveTraining}
Let us now address how to ideally choose the training sets for the forward-adjoint correction to ensure a good fit of the forward correction $\ForwardCor$ by minimising \eqref{equ:forwardLoss_FBC} and the adjoint correction $\AdjointCor$ with \eqref{equ:adjointLoss_FBC}. To create the training set, there are two possibilities. Either we are given a set of measurements $\{\trni{y}{i}, i=1,\dots,N\}$, or alternatively, if we are given a set of samples in image space 
$\{\trni{x}{i}, i=1,\dots,N\}$, then we need to create a corresponding set of measurements by applying the accurate model $\trni{y}{i} = A\trni{x}{i} + \trni{e}{i}$ with the addition of noise $\trni{e}{i}$. 
Either way, given the set of measurements $\trni{y}{i}$ we need to train $\ForwardCor$ and $\AdjointCor$ on a meaningful starting point for the gradient descent to solve the variational problem; a natural candidate would be to choose the backprojection $\ittrni{x}{i}{0} = \ApproxOp^* \trni{y}{i}$. 


Training the corrected operators $\CorrectedOp$ and $\CorrectedAd$ with the samples $\{(\ittrni{x}{i}{0}, \TrueOp \ittrni{x}{i}{0})\}$ only yields
operator corrections that approximate $\TrueOp$ and $\TrueOp^*$ well for samples $x$ that 
are close to backprojections of measurements.
However, the purpose of this 
paper is to learn a correction of $\ApproxOp$ that can be used \textit{within the variational problem} to obtain a solution close to the one 
obtained using the accurate operator $\TrueOp$. We observe that training $\CorrectedOp$ on the backprojections $\ittrni{x}{i}{0} = \ApproxOp^* \trni{y}{i}$ only is not sufficient to achieve this goal. While this leads to $\CorrectedOp$ being a good approximation to $A$ for the first iterates in the gradient descent scheme, 
the approximation quality tends to deteriorate for later iterates, making $\CorrectedOp$ not a good appproximation to $\TrueOp$ anymore. Such a behaviour is in fact what one would heuristically 
expect, as $\CorrectedOp$ has never been trained on later iterates to match the 
accurate operator. 

This connects to the assumptions made in the convergence Theorem \ref{thm:GradientConvergence}, where we assume low approximation error for both the forward and the adjoint at \textit{all} iterates of the gradient descent scheme. We hence need to ensure a uniformly low approximation error at any iterate to be able to guarantee convergence and it is in particular not sufficient to ensure low approximation error at the initial point of the minimisation of the variational problem only.


A natural solution to mitigate this problem is to include later iterates of the 
variational problem into the training samples for the corrected operator. More 
precisely, given some weights $\Theta$ of the correction operator, denote by $\{\ittrni{x}{i}{n}\}$ the iterates obtained following the dynamics
\begin{align}
\label{eqn:DataFid}
\ittrni{x}{i}{n+1} = \ittrni{x}{i}{n} - \mu \left[ \CorrectedAd(\CorrectedOp(\ittrni{x}{i}{n}) - \ittrni{y}{i}{n}) + \lambda \nabla R(x) \right],
\end{align}
where $\mu$ denotes the step size. We add these samples to the original training set $\{(\trni{x}{i}, A\trni{x}{i})\}$, i.e. we also train on $\{(\ittrni{x}{i}{n}, A\ittrni{x}{i}{n})\}$ for all $n<N_{\mathrm{iter}}$ and $i$. Here $N_{\mathrm{iter}}$ is the maximal number of gradient descent steps we take. This allows us to ensure that the corrections $\CorrectedOp$, as well as $\CorrectedAd$ for the forward-adjoint method, are fit consistently well at any iterate $\ittrni{x}{i}{n}$ of the gradient descent dynamics.

A major drawback of this approach is the additional computational burden it incurs during training. Obtaining the iterates of the minimisation to solve the variational problem requires performing the minimisation at training time. \revisionTwo{To reduce the additional computational burden one can make use of the fact that the gradient of the data term for the learned operator correction $\CorrectedOp$ has to be computed for two different purposes. Firstly, it is used to perform minimisation over the variational functional and secondly to further train the $\CorrectedOp$ to better match the accurate operator.} One can hence perform this computation only once, using it for both purposes. This reduces computational costs particularly when training on every iterate of the minimisation over the variational functional, in which case little overhead costs compared to regular training is inflicted.

Additionally, the trajectory \eqref{eqn:DataFid} depends on the network weights $\Theta$. The training samples can hence change during training and convergence is not clear \emph{a-priori}. Empirically, we find that training on the full trajectory $(\ittrni{x}{i}{n}, A\ittrni{x}{i}{n})$ for $n<N_{\mathrm{iter}}$ from the beginning tends to be unstable, as this will lead to most training samples differing greatly from both the original training distribution as well as the accurate trajectory we are finally interested in. There are however two effective solutions to this problem: First, one could alternatively train on the trajectory obtained when using the accurate operator $A$, avoiding instabilities in the beginning of training. This, however, could lead to errors accumulating during training. We found that the most effective solution is to have $N_{\mathrm{iter}}$ increase from $1$ to some $N_{\rm max}$ during training. With this approach, we start off by training on the original samples $\ittrni{x}{i}{0}$ only and then add in more samples from the trajectory as training proceeds. \revision{We have noticed that once trained on backprojections, adding later iterates to the training set does not change the behaviour of the learned correction on backprojections by much. In this sense, one can interpret the latter approach to recursive training as gradually extending the domain the correction is valid on, without considerably changing the behaviour of the correction on the part of the image domain that it is already valid on. This heuristically explains why recursive training can be performed very stably when gradually increasing $N_{\mathrm{iter}}$.}

\begin{figure}[t!]
	\centering
	\begin{picture}(350,130)
	\put(0,5){\includegraphics[width = 350 pt]{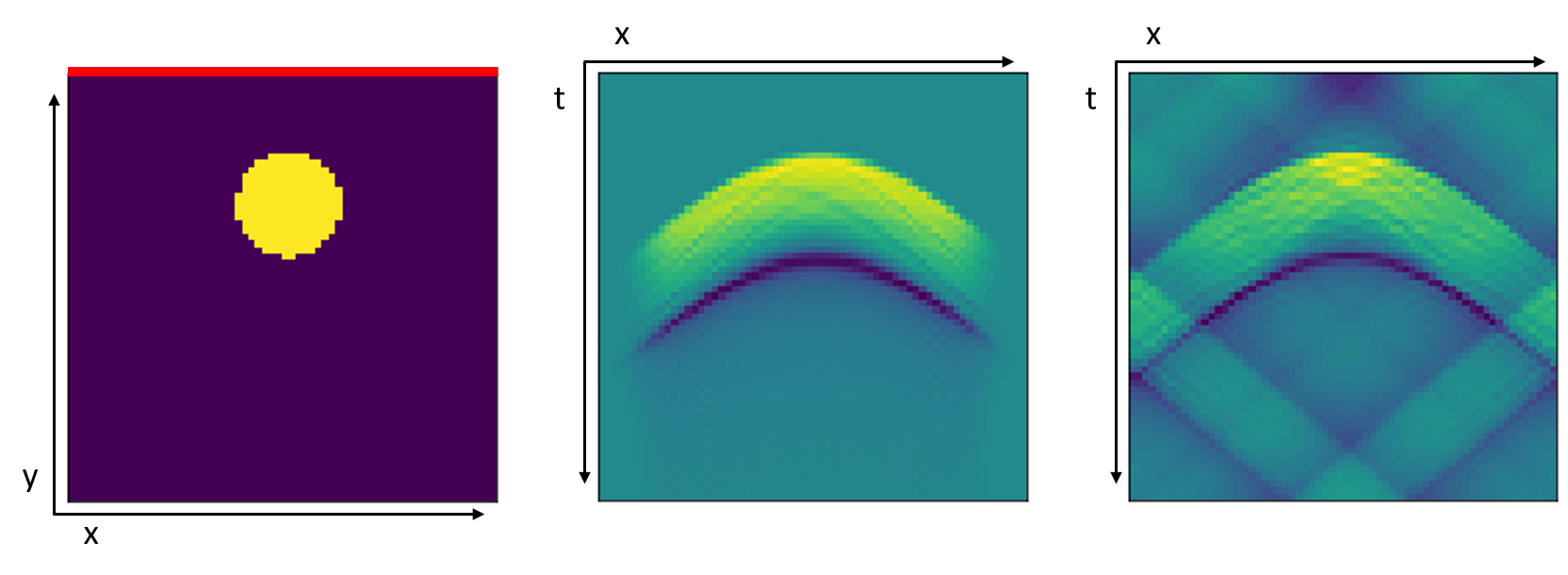}}
	\put(35,0){Phantom: $x$}
	\put(140,0){Accurate data: $Ax$}
	\put(248,0){Approximate data: $\ApproxOp x$}
	\end{picture}
	\caption{\label{fig:dataPATillustration} Illustration of the \revision{limited view} imaging scenario under consideration. Left: numerical phantom with a line detector (red line). Middle: ideal data from the accurate forward model. Right: data obtained with an approximate model with clearly visible aliasing artefacts.}
\end{figure}
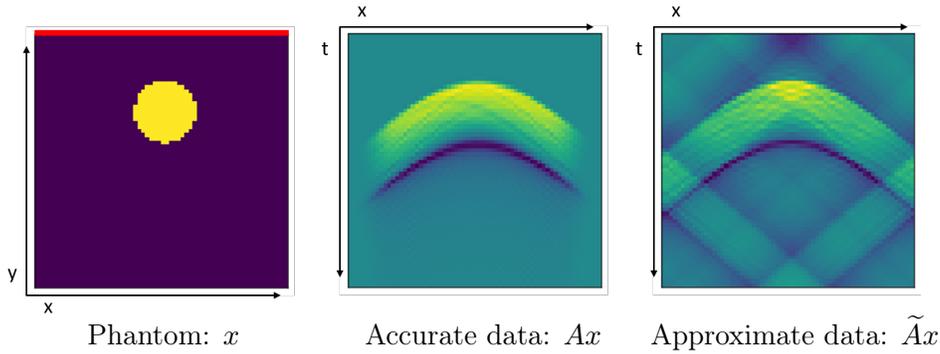

\subsection{Experiment design}
For a practical application we consider photoacoustic tomography (PAT) in two dimensions; for more details on PAT see \cite{Beard2011} and the discussion in Appendix \ref{sec:appxPAT}. 
Here, the measurement data is given as a set of time-series  \revision{in a limited view geometry measured with a line detector at the surface},
which we visualise as a space-time image in Figure \ref{fig:dataPATillustration}. \revision{In this limited view scenario, the reconstruction task is already a very challenging inverse problem in itself even with the accurate operator available, we refer to \cite{kuchment2008mathematics,Xu2004} for details.}
Here, the accurate model $A$ is given by a pseudo-spectral time-stepping model \cite{treeby2010,Treeby2012}, whereas the approximate model $\ApproxOp$ is given by a regriding and Fast Fourier Transform which neglects the effect of singularities and introduces systematic errors in the forward mapping \cite{Cox2005,Koestli2001}.
In particular, to avoid singularities in the approximate model we threshold incident waves with an angle up to $\theta_{\mathrm{max}}=60^\circ$ from normal incidence\revision{, which means that this part of the data is inevitably lost.
	Nevertheless, the approximate forward model still exhibits strong aliasing artefacts, as} can be clearly seen in Figure \ref{fig:dataPATillustration} 
indicating that this application is an ideal candidate for this study. 
For more details on the models, we refer to the discussion in Appendix \ref{sec:appxPAT}. \revision{We developed the majority of code in Python using the TensorFlow package and using the k-Wave MATLAB (R2018b, The MathWorks Inc., Natick, MA) toolbox \cite{treeby2010} for some calculations concerning the accurate operator. We used a single Quadro P6000 to conduct the experiments.
}


\paragraph{Model corrections under consideration}
We evaluate the forward only method with a gradient penalty term as described in Section \ref{sec:forwardCorrection} as well as the forward-adjoint approach as outlined in Section \ref{sect:ForwardBackwardCorr} \footnote{Code is available at \url{https://github.com/lunz-s/ModelCorrection}}. For both of these methods, we conduct experiments with a model trained on back-projected measurements only and with a model that has been trained using recursive training \revision{(Section \ref{sec:RecursiveTraining})}. As a baseline method, we compare to the widely used AEM approach as outlined in Section \ref{sect:AEM}, a linear approach to model correction. We finally compare to reconstructions obtained with the uncorrected operator as well as to the reconstruction the accurate operator yields. This allows to assess how well various correction approaches are able to correct the shortcomings of the uncorrected operator.


\paragraph{Measurement setup}
We consider a limited view problem in this study, where measurements are only taken \revision{on top of the target} with a line detector, as indicated in Figure \ref{fig:dataPATillustration}. \revision{In particular, we consider an image size of $64\times 64$, the measurements are taken with a line detector of the same width as the target and $t=64$ time points, resulting in a measurement space of same size, i.e. $64\times 64$. The detector is modelled as a Fabry-P{\'e}rot  sensor \cite{zhang2008backward} with wide bandwidth and no directivity.} 
Since both image and data space can be represented as a two-dimensional image, it is reasonable to use the same network architecture for both spaces.


\paragraph{Training samples}

\begin{figure}
	\centering
	\includegraphics[width=\textwidth]{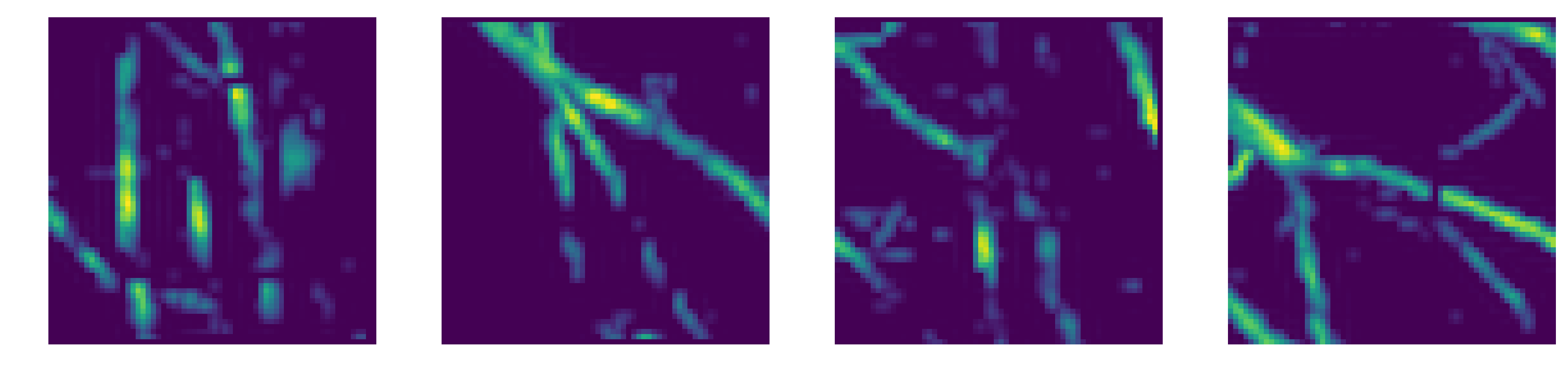}
	\caption{Examples from the vessel set used for training of the model correction. The phantoms were obtained from segmented CT scans to provide a realistic ground-truth image for photoacoustic imaging of vessel structures.}
	\label{fig:Samples_vessels}
\end{figure}

For the evaluation of the various model correction methods, we utilise two different sets of samples. Firstly, a simple synthetic set of 'ball' images, consisting of circles of varying intensity in $[0.75,1]$, with fixed radius, but random location on an empty, zero intensity background. We employ a total of $4096$ ball samples for fitting the correction and an additional $64$ for evaluation. An example of a \revision{`ball'} image and the corresponding data are illustrated in Figure \ref{fig:dataPATillustration}.
Secondly, a realistic vessel set that has been obtained by segmenting vessels from 3D CT scans to provide realistic phantoms, see \cite{Hauptmann2018} for details. For this study, the 3D volumes have been projected to two dimensions by a maximum intensity 
projection and subsequently cropped to the intended target size; we note that all samples are normalised between $[0,1]$. Examples of the obtained vessel phantoms are displayed 
in Figure \ref{fig:Samples_vessels}. We use $2760$ unique vessel phantoms for training, augmented by a rotation by $90^\circ$ for a training set of $5520$ samples in total. We evaluate on a separate test set containing $64$ samples. All phantoms had a resolution of $64^2$ and resolution in data space is the same for both, correct and approximate model. The phantoms are used to generate synthetic measurements $y^i := \TrueOp x^i + e^i$ by applying the accurate operator $\TrueOp$ and adding Gaussian white noise at $1\%$ of the maximum value in measurement space.

\paragraph{Training Scheme}

For every measurement $\trni{y}{i}$, we compute $\ittrni{x}{i}{0} := 4 \cdot \ApproxOp^* y$ as an initial reconstruction. 
We choose to rescale the adjoint $\ApproxOp^* y$ by a factor of $4$ as in our measurement setup we typically have $\|\TrueOp x\|_Y \approx \frac{1}{2} \|x\|_X$ and $\|\TrueOp^*y\|_X \approx \frac{1}{2} \|y\|_Y$. This is due to the fact that we measure along a line on one side of the object only, hence recording only half the energy emitted on the measurement device. This \revision{ensures that the average intensity of the backprojection roughly matches the one of both the ground truth and the minimiser of the variational functional. It }allows to keep the norm of the reconstruction approximately stable throughout solving the variational problem \eqref{eqn:VariationalProblem} and hence makes operator approximations more robust throughout the trajectory of minimising \eqref{eqn:VariationalProblem}.

Given a set of training samples $\trni{y}{i}$, we then train the forward approximation with the loss term
\begin{align}
\label{eqn:ForwardLoss}
\sum_i \underbrace{\left\|\ForwardCor (\ApproxOp \ittrni{x}{i}{0}) - \TrueOp \ittrni{x}{i}{0} \right\|_Y}_{\text{Forward Loss}} + \underbrace{\left\|\left(\TrueOp^*- \ApproxOp^* \left[D \ForwardCor(\ApproxOp x_o^i) \right]^*  \right) \left( \ForwardCor(\ApproxOp \ittrni{x}{i}{0})-\trni{y}{i} \right)\right \|_X}_{\text{Adjoint Loss}},
\end{align}
weighting the forward and adjoint loss equally.
In the case of a forward-adjoint correction, the forward approximation is trained using the loss
\begin{align}
\label{eqn:ForwardLossFB}
\sum_i \left\|\ForwardCor (\ApproxOp \ittrni{x}{i}{0}) - \TrueOp \ittrni{x}{i}{0} \right\|_Y,
\end{align}
while the adjoint is trained with the loss
\begin{align}
\label{eqn:AdjointLossFB}
\sum_i \left\| \left( \AdjointCor \circ \ApproxOp^* -  A^* \right) \left(\ForwardCor (\ApproxOp \ittrni{x}{i}{0})-\trni{y}{i} \right) \right\|_X.
\end{align}

Note that the quasi-adjoint of the approximate operator $\CorrectedAd :=  \AdjointCor \circ \ApproxOp^*$ as well as the adjoint of the forward approximation in \eqref{eqn:ForwardLoss} is evaluated in direction $r := \ForwardCor( \ApproxOp \ittrni{x}{i}{0})-\trni{y}{i}$. This loss is chosen to be consistent with the terms arising during a gradient-descent based optimisation of \eqref{eqn:VariationalProblem}, as shown in the previous chapters.

If recursive training is applied, we additionally compute the iterates of a gradient-descent scheme on the penalty functional
\begin{align}
\label{eqn:VariationalProblem}
\argmin_x \| \CorrectedOp(x) - y^i \| + \lambda R(x).
\end{align}
All losses are summed over the later iterates $\ittrni{x}{i}{n}$ with $n \geq 0$, instead of taking the initial point $\ittrni{x}{i}{0}$ only. To make recursive training stable, the number of recursive steps considered during training is gradually increased to the maximal value, instead of beginning by training on the full trajectory from the start as outlined in Section \ref{sec:RecursiveTraining}.


\paragraph{Network Details}
The networks $\ForwardCor$ and $\AdjointCor$ are built with a U-Net \cite{Ronneberger2015} architecture, that has been particularly popular in the image reconstruction community including applications to PAT  \cite{Antholzer2019deep,davoudi2019deep,guan2019fully} and other modalities \cite{hamilton2018deep,hauptmann2019real,Jin2017}. We follow the standard architecture with $4$ downsampling and the same amount of upsampling blocks, each containing two convolutional layers with filters of size $5\times 5$. We employed average pooling for downsampling and transpose convolutions for upsampling layers. We note, that the proposed framework is agnostic to the employed architecture, we expect similar results with other sufficiently expressive network architectures.

\paragraph{Solving the variational problem}
We employ gradient descent with a fixed step size \revision{of 0.2 for all experiments} to solve the variational problem \eqref{eqn:VariationalProblem}, which we have seen can lead to a near-optimal reconstruction given sufficient approximation quality in Section \ref{sec:ConvergenceAnalysis}. We additionally add a positivity constraint $x_n \geq 0$ everywhere to the minimisation that we incorporate using projected gradient descent. This means we cut the negative part of every iterate to $0$ everywhere, as negative values are non-physical. 

As regularisation functional $R$ we choose the pseudo-Huber varation functional
\begin{align}
R(x) := \sum_{i,j} \delta \left[ \sqrt{ 1 + \frac{1}{\delta^2} [(x[i+1, j] - x[i, j])^2 + (x[i, j+1] - x[i, j])^2]} - 1 \right] 
\end{align}
to reconstruct $x \in \mathbb{R}^{64 \times 64}$. Here $x[i,j]$ denotes the pixel of $x$ at location $i$ along the vertical and $j$ along the horizontal axis. This functional approximates the $L^2$-norm of the gradient of the reconstruction for small values and the $L^1$-norm of for large values of the gradient, coinciding with total variation (TV) in the limit $\delta \to 0$. The parameter $\delta$ specifies the characeristic length at which the behaviour of the regularisation functional changes from approximating $L^2$ to $L^1$. We chose $\delta = 0.01$ for all experiments. \revision{We remark that this functional is strongly convex on all bounded domains for all $\delta>0$, with the strong convexity constant depending on $\delta$ and the diameter of the imaging domain. The latter is in our case specified by the constraint $x[i,j] \in [0,1]$.}

The regularisation parameter $\lambda$ is tuned for every experiment and baseline individually via a grid search over a logarithmically evenly spaced grid with grid points being a factor of $log(10)$ apart. The best parameter was chosen in terms of $L^2$ distance to the ground truth image.

\section{Computational results}\label{sec:computationalResults}


\paragraph{Synthetic ball phantoms}

To evaluate the proposed approaches we solve the variational problem employing the various approaches for model correction for a set of samples generated from a test set that is different from the samples used for fitting the correction. We use the same Huber regularisation functional and regularisation parameter as discussed in the last paragraph.

First, we investigate the correction accuracy in terms of the alignment of the gradient of the data fidelity term with the accurate gradient $\TrueOp^*(\TrueOp x_n - y)$ throughout the minimisation of the variational functional in Figure \ref{fig:Angles_balls}. As a notion of alignment we consider 
\begin{align}
\label{eqn:AlignmentExperimental}
\cos \Phi_v(x) = \frac{\left\langle \TrueOp^* \Big(\TrueOp x_n - y \Big), \left(\AdjointCor\circ \ApproxOp^*\right) \left( \ForwardCor(\ApproxOp x) - y \right) \right\rangle_X}{\left\| \TrueOp^*\Big(\TrueOp x_n - y\Big)\right\|_X \left\|\Big(\AdjointCor\circ \ApproxOp^*\Big) \left( \ForwardCor(\ApproxOp x) - y \right)\right\|_X},
\end{align}
in the case of the forward-adjoint method. For the forward only and AEM methods, the expression $\big(\AdjointCor\circ \ApproxOp^*\big) \big( \ForwardCor(\ApproxOp x) - y \big)$ is replaced by the appropriate gradient of the corrected data fidelity term. Equation \eqref{eqn:AlignmentExperimental} is a slight deviation from \eqref{eqn:AligenementVariational} used in the theory section. This is to ensure good comparability with the baseline AEM and better \revision{interpretability by rescaling the alignment with the norm of the approximate gradient. This also makes different choices of regularisation parameters more comparable. In the theory section we instead rescale with the norm of the accurate gradient only, making the proofs more straight forward.}

We note that all correction methods apart from the AEM approach start at a high alignment of $>0.8$ at the first iterate. However, only the forward-adjoint based methods are able to achieve an alignment of $>0.95$ at the first iterate. Forward only approaches that rely on fitting a correction in measurement space only are limited by the range of the adjoint $\ApproxOp^*$ as discussed in Section \ref{sect:ForwardBackwardCorr}. 

However, the  alignment starts decreasing rapidly over the minimisation of the variational problem, dropping below $0$ for the forward-adjoint method before the $200$th iterate. The recursive versions of the forward and forward-adjoint methods, as discussed in Section \ref{sec:RecursiveTraining}, are able to mitigate some of this shortcoming. While the alignment between accurate gradient and the correction also declines throughout the minimisation of the variational problem when employing recursive training, the decline is significantly less steep and occurs at a later stage of the minimisation. We also note that the alignment never drops under $0.2$ for recursively trained corrections. 
\begin{figure}
	\centering
	\begin{subfigure}{.49\textwidth}
		\includegraphics[scale=0.45]{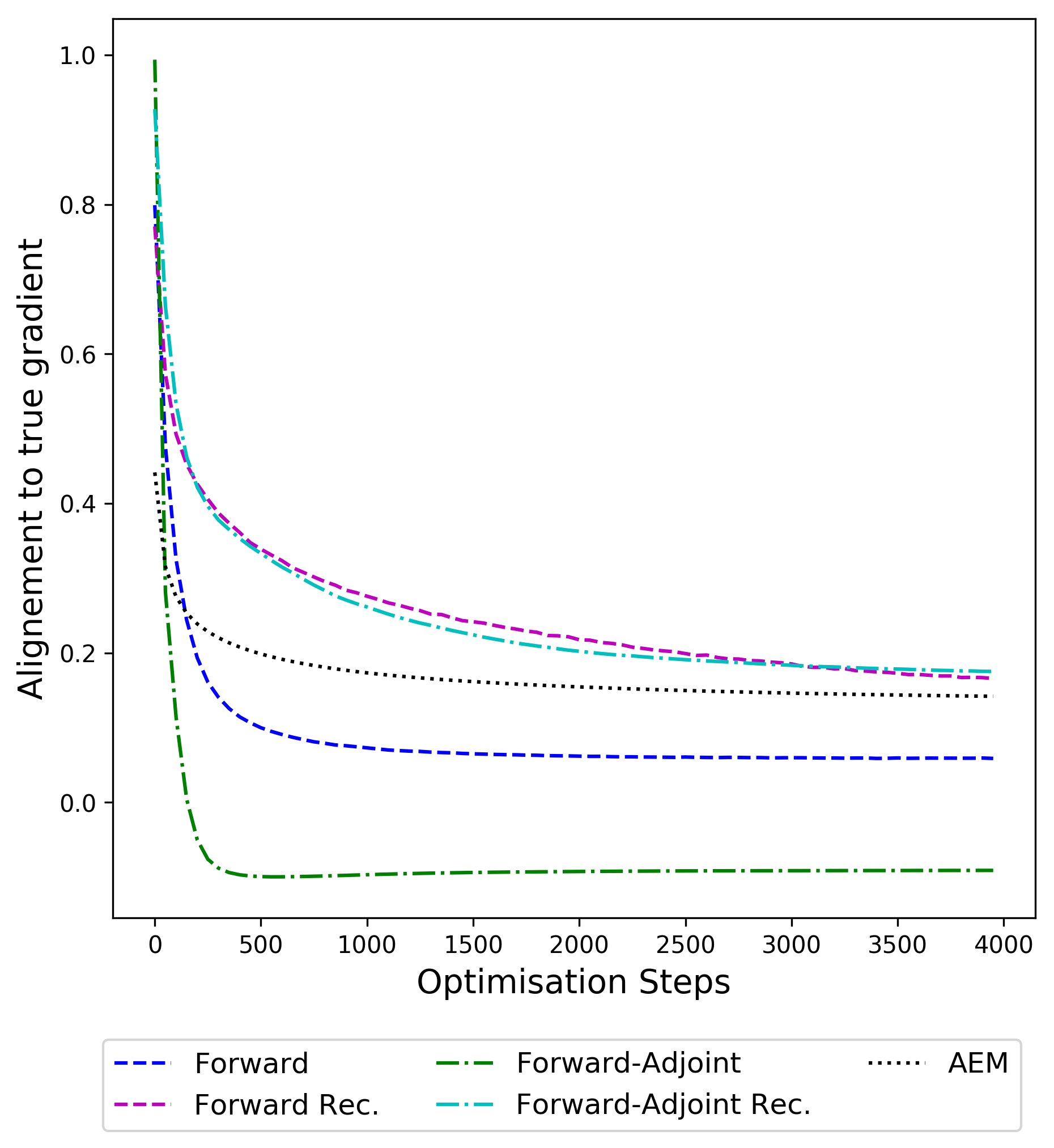}
		\caption{Full trajectory}
		\label{fig:Angles_balls_A}
	\end{subfigure}
	\begin{subfigure}{.49\textwidth}
		\includegraphics[scale=0.45]{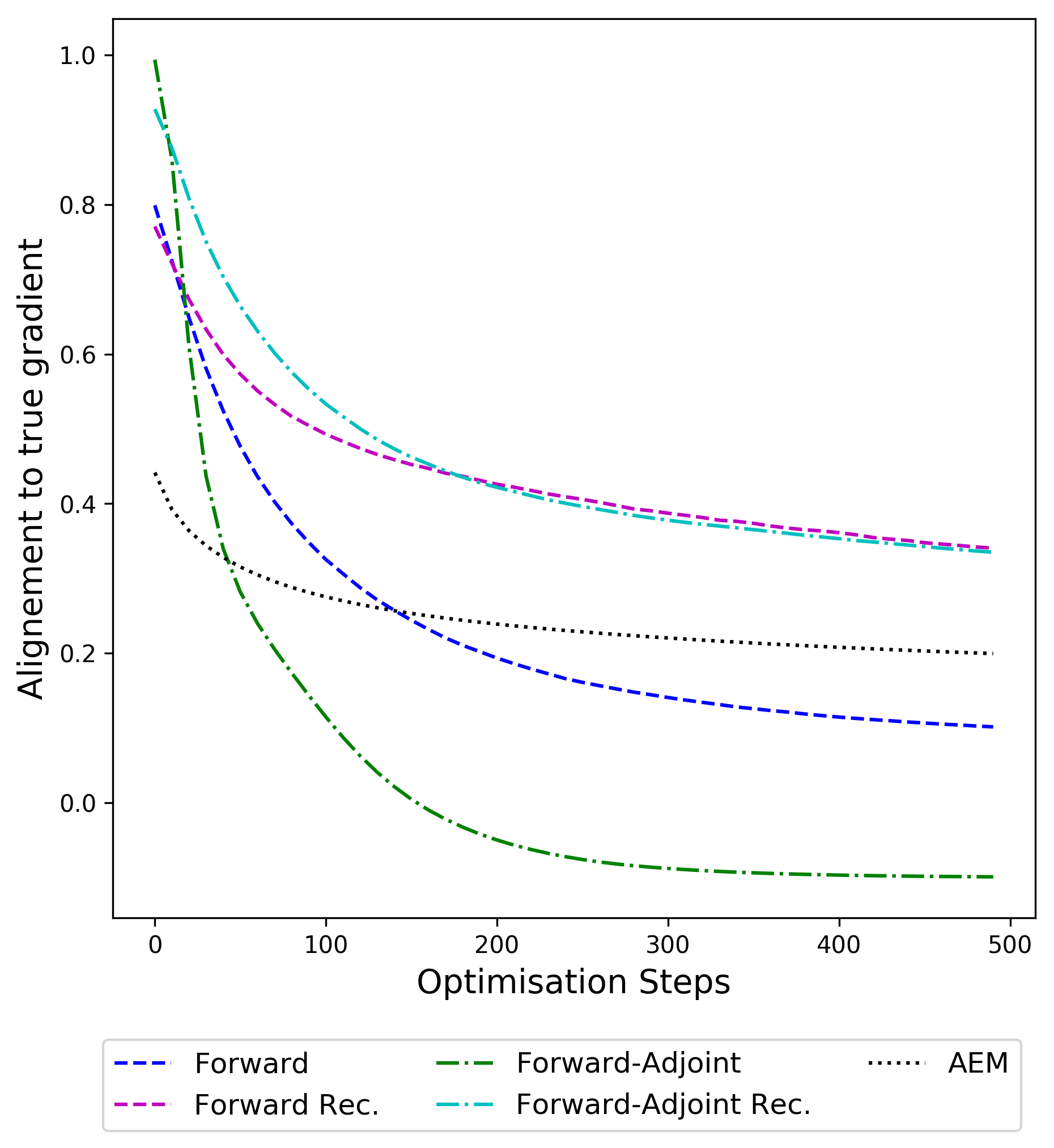}
		\caption{First $500$ steps}
		\label{fig:Angles_balls_B}
	\end{subfigure}
	\caption{Alignment \eqref{eqn:AlignmentExperimental} of approximate gradient to the gradient of the accurate data term $\TrueOp^*(\TrueOp x_n - y)$ for each approach on the ball test set of 64 samples. The alignment is recorded over all minimisation steps for solving the associated variational problem.
		On the left (a) for the full trajectory and on the right (b) for the first 500 steps.}
	\label{fig:Angles_balls}
\end{figure}

The benchmark AEM method is not able to correct the gradient as accurately as any of the methods we discussed for the first iterates of the variational problem. However, it does not exhibit a decline of the alignment as drastic as any of the other methods throughout the minimisation process. This can be explained by the lower expressive power of AEM compared to the corrections based on neural networks that does not allow the method to fit the accurate gradient as well for early iterates but prevents overfitting on later iterates, leading to the method being stable throughout the minimisation of the variational functional.

The different behaviour of forward and forward-adjoint methods as well as their recursive counterparts is investigated in Figure \ref{fig:LossFunctions_balls}. We note that in terms of the forward approximation error, applying recursive training makes the key difference in terms of keeping a low error throughout gradient descent. For the adjoint approximation error we note that methods based on the forward scheme that fit a single operator are not able to achieve low error, even at the first iterate due to the fundamental limitations of the method. Forward-adjoint methods on the other hand are able to fit the accurate adjoint well at the first iterates, but also suffer from deteriorated approximation quality for later steps.

\begin{figure}
	\centering
	\begin{subfigure}{.49\textwidth}
		\includegraphics[scale=0.4]{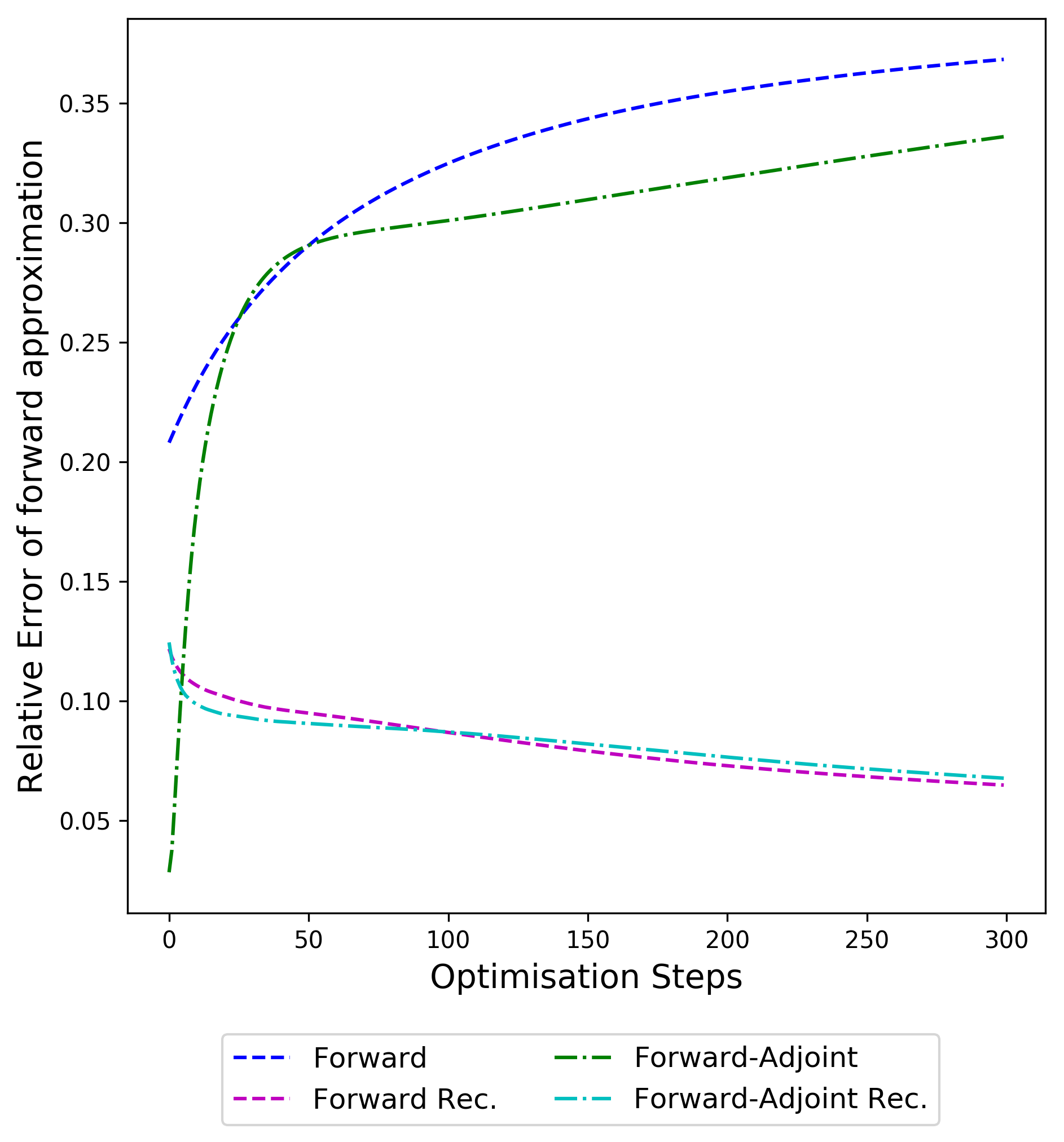}
		\caption{Relative approximation error of forward operator}
	\end{subfigure}
	\begin{subfigure}{.49\textwidth}
		\includegraphics[scale=0.4]{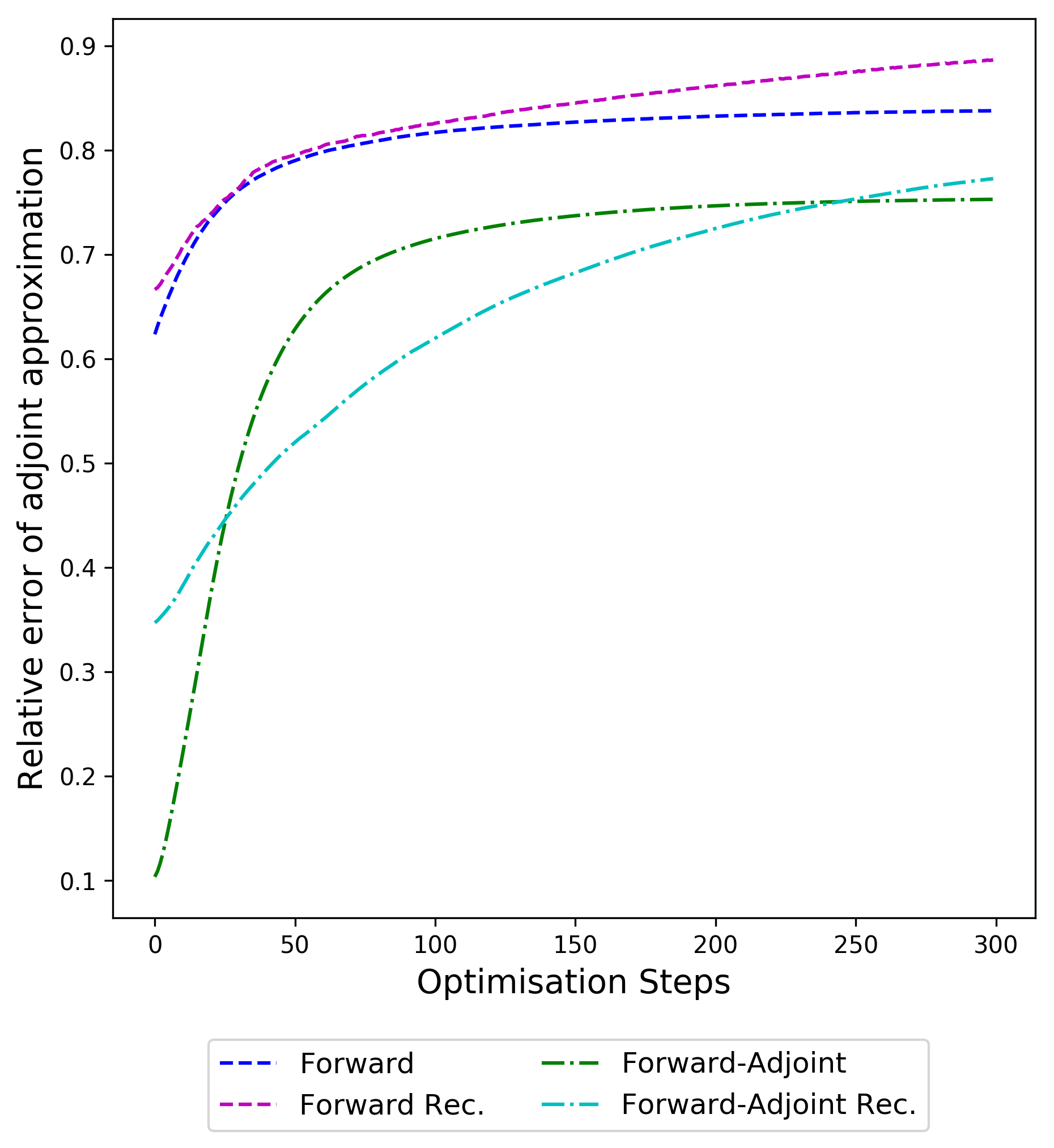}
		\caption{Relative approximation error of adjoint operator}
	\end{subfigure}
	\caption{Approximation error of the model correction compared to the accurate operator on the ball test set of 64 samples, tracked throughout the first $300$ steps of the gradient descent scheme. Left (a): relative error of the forward approximation as defined in \eqref{eqn:ForwardLossFB}.  Right (b): relative error for the adjoint, as defined for the forward only in equation \eqref{eqn:ForwardLoss} and for the forward-adjoint method in  \eqref{eqn:AdjointLossFB}.}
	\label{fig:LossFunctions_balls}
\end{figure}

In Figure \ref{fig:Data_Term_balls}, we see evolution of the data term $\|\TrueOp x_n - y\|_Y$ evaluated using the accurate operator $\TrueOp$ in order to test if the corrections minimise the original variational problem. We note that both recursive methods are able to effectively minimise the data term quickly, with both converging stably to their respective minimal value. This empirical observation shows that the learned reconstructions in fact lead to a variational energy that satisfy Lemma \ref{lem:LowEnergyCloseness} to ensure closeness of minimiser. 
We note that forward-adjoint recursive is able to achieve a lower data loss than its forward only counterpart, which is consistent with the behaviour observed in Figure \ref{fig:Angles_balls}. It is interesting to note, that both methods are able to minimise the accurate data term significantly better than the baseline AEM. When omitting recursive training both the forward only and the forward-adjoint algorithm are not able to minimise the accurate data term well.

Finally, we evaluate the model correction in terms of the distance of the reconstruction to the ground truth image, measured by the relative $L^2$ error shown in Figure \ref{fig:Quality_balls}. We note that all approximation approaches outperform the uncorrected operator in this metric. Both corrections, forward and forward-adjoint, without recursive training lead to a decrease in reconstruction error reconstruction quality for the first $300$ optimisation steps, stagnating or even deteriorating afterwards. This is again consistent with the findings in Figure \ref{fig:Angles_balls}, which show that the gradient generated by these methods does not align with the accurate gradient any more at this point of the minimisation. The recursive counterparts of the forward and forward-adjoint method produce considerably better results, with the recursive forward-adjoint method generating reconstructions that are nearly of the same quality as the ones generated by the accurate operator. The baseline with AEM is converging more slowly than any of the other methods but is able to produce high-quality results after $4000$ gradient descent steps that are on par with the forward recursive method, but are significantly outperformed by the recursive forward-adjoint method.

\begin{figure}[!htb]
	\centering
	\begin{minipage}{0.45\textwidth}
		\centering
		\includegraphics[scale=0.4]{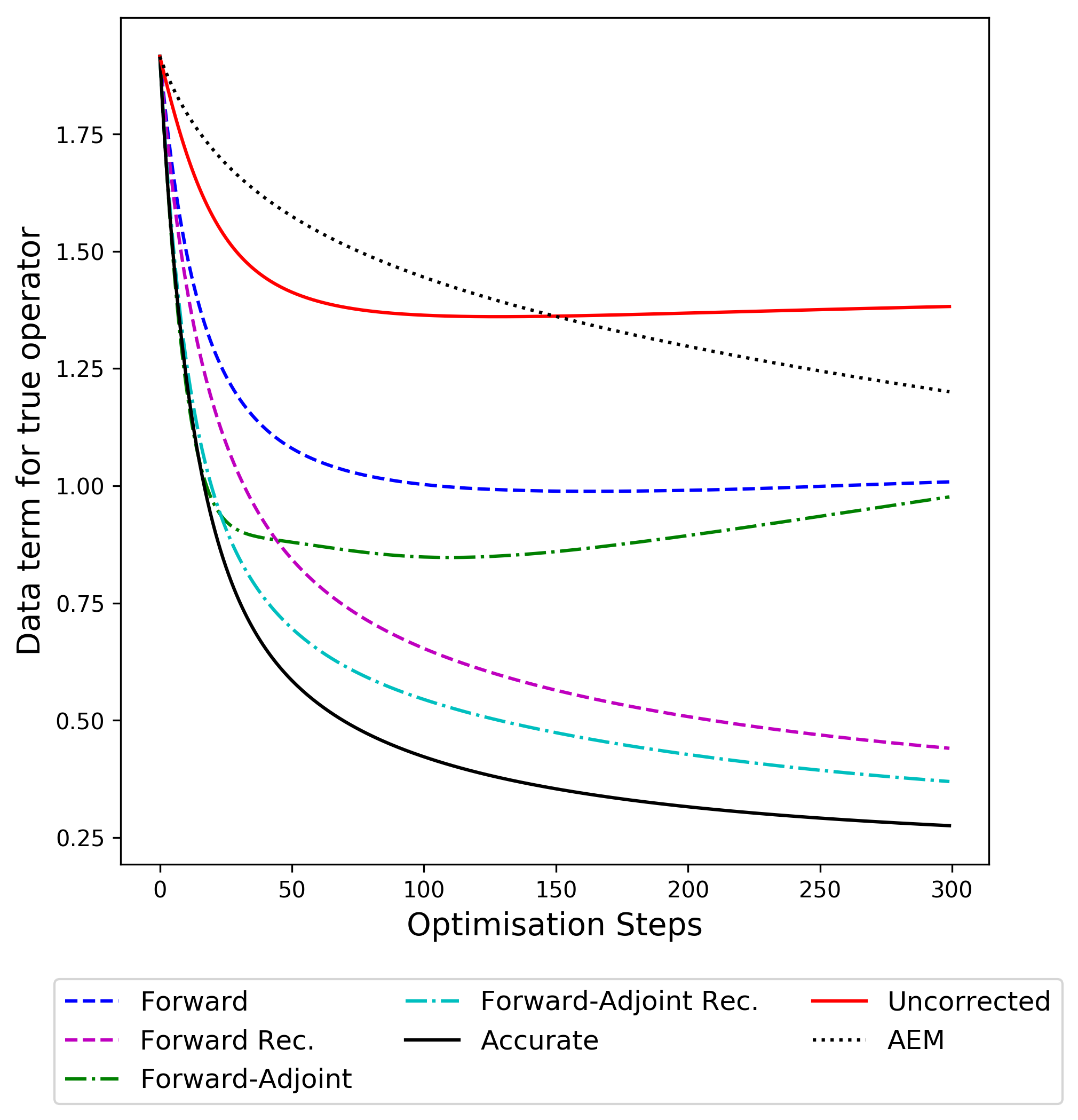}
		\caption{True data term $\|Ax_n - y\|_Y$ evaluated for all methods on the ball test set of 64 samples, tracked throughout the gradient descent scheme.}
		\label{fig:Data_Term_balls}
	\end{minipage}
	~
	\begin{minipage}{.45\textwidth}
		\centering
		\includegraphics[scale=0.4]{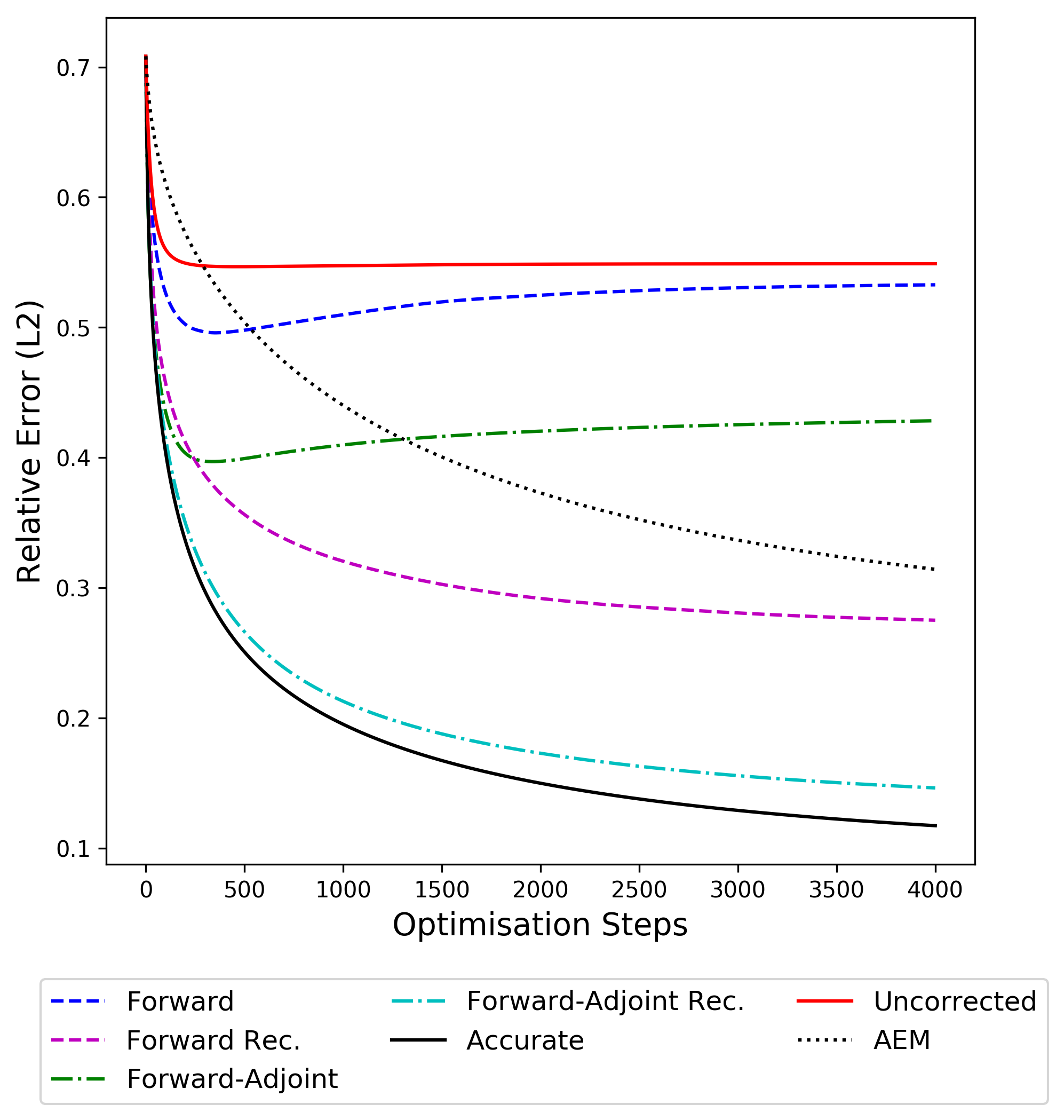}
		\caption{Relative reconstruction error (L2) for all methods on the ball test set of 64 samples, tracked throughout the gradient descent scheme.}
		\label{fig:Quality_balls}
	\end{minipage}%
\end{figure}

For a qualitative evaluation, we show obtained reconstructions in Figure \ref{fig:Reconstructions_balls} for all methods discussed and two samples with different 
behaviour. In the first example, where the ball is close to the line detector, we note 
that all methods are able to correct the errors introduced by the approximate operator to some extent. However, both the forward and forward-adjoint method introduce background 
artefacts when not trained recursively. These artefacts disappear when recursive training is applied, leading to near perfect reconstructions. Compared to AEM as baseline, which 
is able to correct the approximate operator without introducing background artefacts, 
the correction by AEM introduces blurred edges of the ball that are not observed by any of the neural network based corrections we are investigating. The second sample is particularly more challenging, with the ball being far from the detector exhibiting 
stronger limited-view artefacts and consequently the approximate operator introduces 
severe artefacts if uncorrected. For the corrections without recursive training we see 
again that both approaches, forward and forward-adjoint, introduce background artefacts. For the forward method, these artefacts can not be suppressed by applying recursive 
training, leaving a severe artefact at the boundary of the domain. Only the recursive 
forward-adjoint is able to produce a reconstruction that is nearly en par with the 
reconstruction obtained with the accurate operator and that does not exhibit any obvious 
artefacts. The baseline with AEM also introduces background artefacts leaking from the 
ball, but those are more structured and less severe than those of all other methods apart from the forward-adjoint recursive approach which gives the best visual results in this 
setting as well. The visual quality of the reconstructions hence coincides with the 
quantitative results discussed in Figure \ref{fig:Quality_balls}.

\begin{figure}
	\centering
	\begin{subfigure}{\textwidth}
		\centering
		\caption{Reconstructions for phantom close to the line detector}
		\includegraphics[scale=0.4]{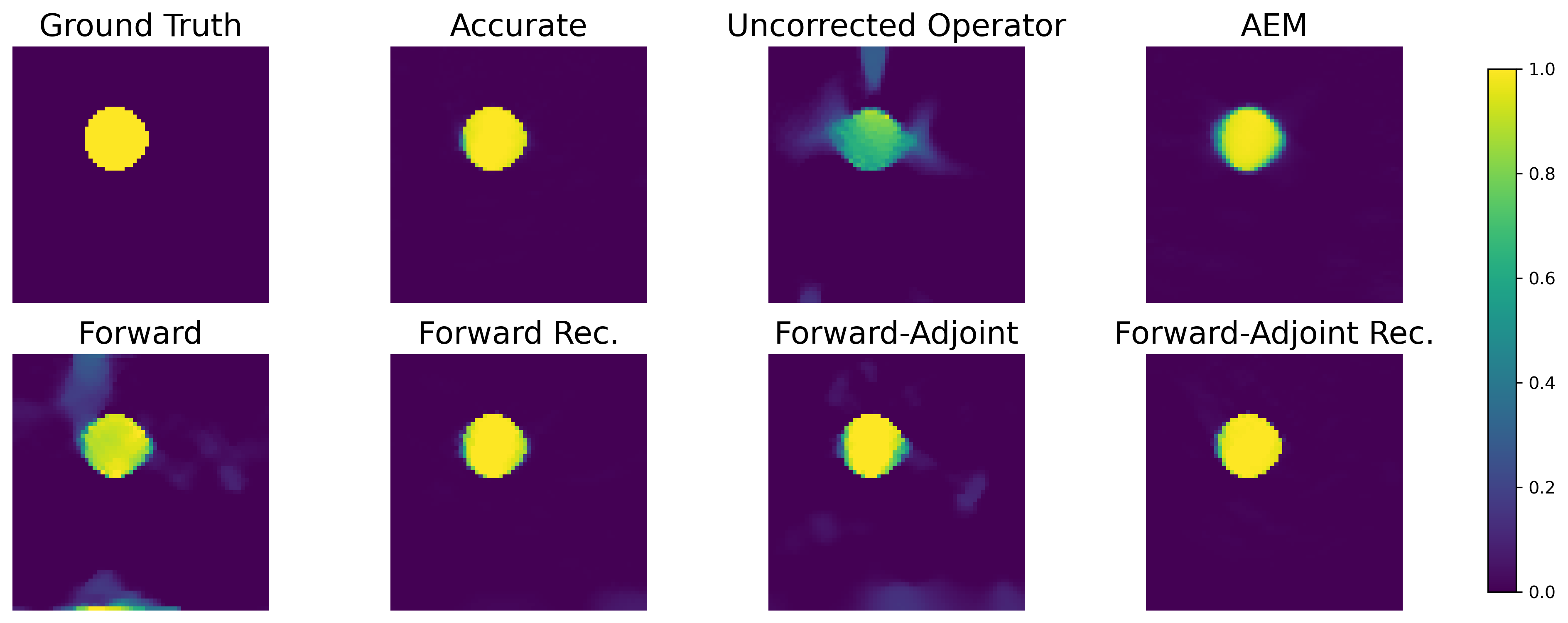}
	\end{subfigure}
	\begin{subfigure}{\textwidth}
		\centering
		\caption{Reconstructions for phantom far from the line detector}
		\includegraphics[scale=0.4]{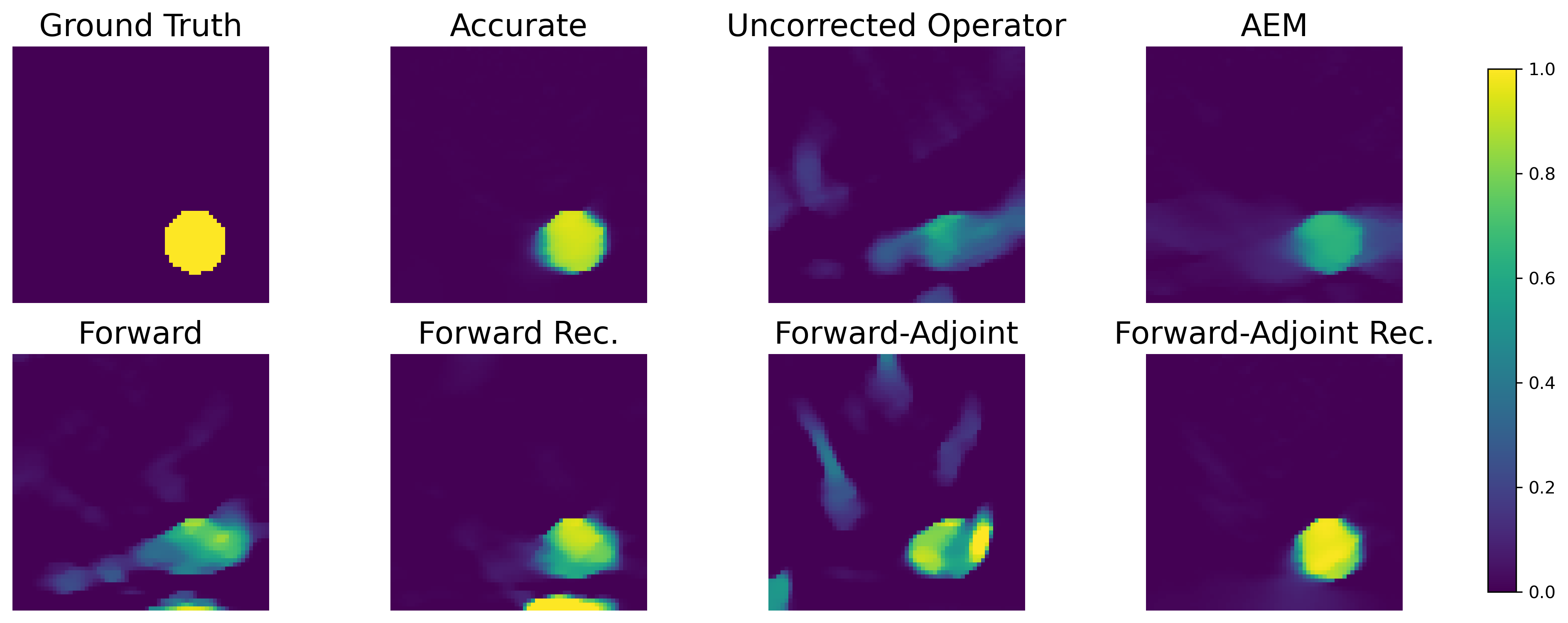}
	\end{subfigure}
	\caption{Reconstructions for the various model correction algorithms for two samples from the ball set. We show the results after $4000$ steps of gradient descent. Huber regularisation is used.
		Top (a): Phantom close to the detector, which corresponds to an easy setting for limited-view PAT. Bottom (b): Phantom far from the the detector, which corresponds to a very challenging setting.   
	}
	\label{fig:Reconstructions_balls}
\end{figure}

\revision{Figure \ref{fig:Measurements_0} visualises the effect of the forward-adjoint recursive approach on the ball images, showing $\TrueOp x_{0}$, $\ApproxOp x_{0}$ and $\CorrectedOp (x_0)$ as well as the gradients of the data term for each of the operators $\TrueOp$, $\ApproxOp$ and $\CorrectedOp$. The visualisations are computed for Sample (b) in Figure \ref{fig:Reconstructions_balls} on the ball samples. We see that the forward-adjoint approach is in fact able to correct for approximation artefacts \textit{both} in the forward operator as well as in its adjoint, leading to a good approximation of the accurate gradient of the data term.}

\begin{figure}
	\centering
	\begin{subfigure}{.8\textwidth}
		\includegraphics[width=\textwidth,trim={2cm 9.5cm 2cm 8.5cm},clip]{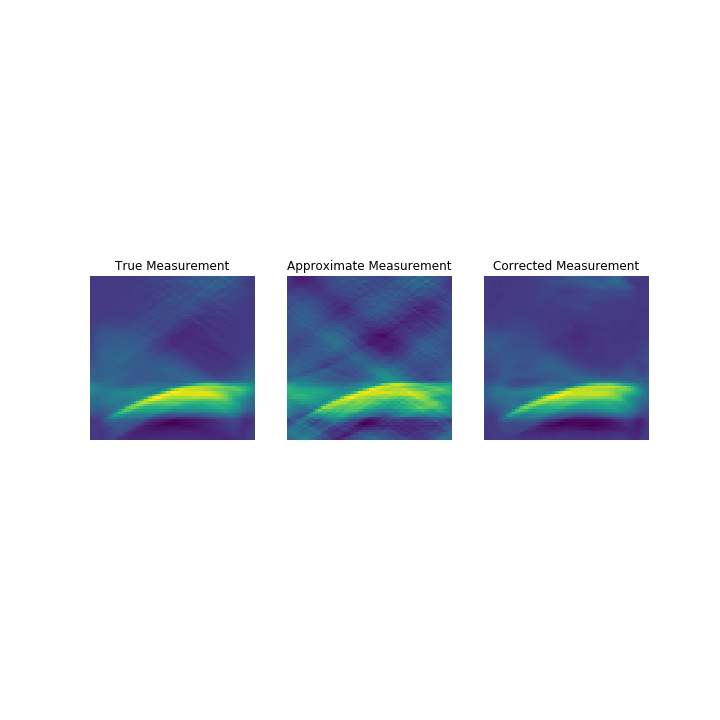}
		\caption{Measurements computed at the initial reconstruction for the exact, approximate and corrected Operators. From left to right $\TrueOp x_{0}$, $\ApproxOp x_{0}$, $\CorrectedOp(x_{0})$}
	\end{subfigure}
	\begin{subfigure}{.8\textwidth}
		\includegraphics[width=\textwidth,trim={2cm 9.5cm 2cm 8.5cm},clip]{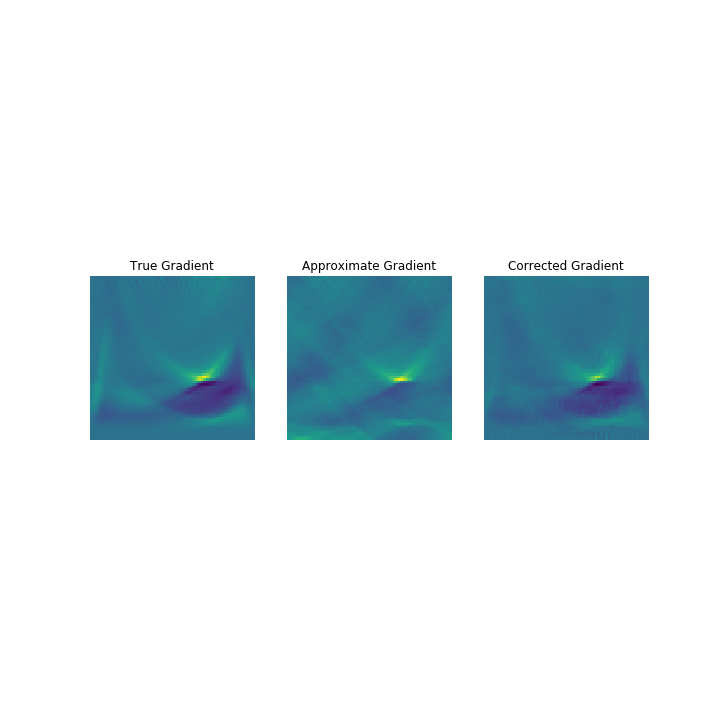}
		\caption{Gradients of the data term, computed at the initial reconstruction for the exact, approximate and corrected Operators.}
	\end{subfigure}
	\caption{\revision{Estimated measurements and gradients at initialisation of the gradient descent scheme for a sample from the ball images.}}
	\label{fig:Measurements_0}
\end{figure}

\paragraph{Vessel phantoms}



The results on the vessel phantoms quantitatively match the overall behaviour observed on the ball set.
The alignment, as shown in Figure \ref{fig:Alignement_vessels}, is again initially higher with forward-adjoint methods achieving higher values as forward only methods. If no recursive training is applied, alignment declines very quickly. AEM is again generating gradients of comparatively low initial alignment, that however stays relatively steady throughout solving the variational problem. We note that the overall alignment is significantly lower than in the case of the ball samples, reflecting the additional difficulty of the vessel set.


The relative error of the reconstructions compared to the ground truth can be seen in Figure \ref{fig:Quality_vessels}. We again see both the forward and forward-adjoint methods fail to improve reconstruction quality further early into the minimisation 
process if recursive training is omitted. In case recursive training is applied, both methods lead to a clear improvement over the uncorrected operator, with the forward-adjoint approach again performing considerably better than the forward only. On the vessel samples we however note a considerably larger gap between the forward-adjoint correction and the accurate operator that is caused by the extremely challenging nature of the vessel set. The AEM baseline converges slowly on the vessels, an indication that the estimated covariance matrix is fairly ill-conditioned. We hence additionally report the reconstruction quality at convergence, which we observed after $20000$ steps of gradient descent. While this is a competitive reconstruction, it is still outperformed slightly by the recursive forward-adjoint method. We remark that we have applied early stopping for all other methods on the vessel samples.

\begin{figure}[!htb]
	\centering
	\begin{minipage}{0.45\textwidth}
		\centering
		\includegraphics[scale=0.4]{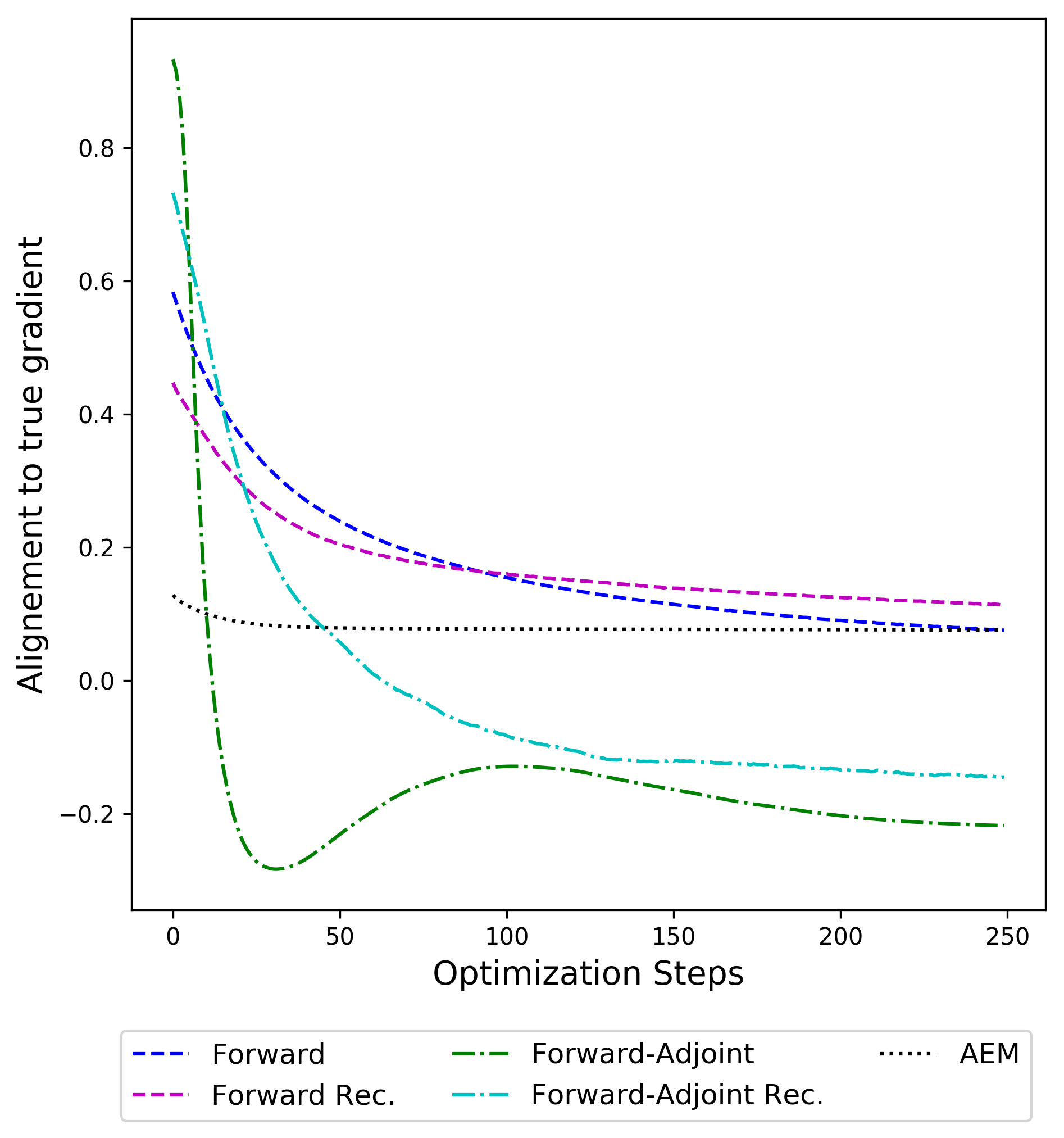}
		\caption{Alignment \eqref{eqn:AlignmentExperimental} of approximate gradient to the gradient of the accurate data term $\TrueOp^*(\TrueOp x_n - y)$ for each method on the vessel test set with 64 samples, recorded over the 250 steps of solving the associated variational problem.}
		\label{fig:Alignement_vessels}
	\end{minipage}
	~
	\begin{minipage}{.45\textwidth}
		\centering
		\includegraphics[scale=0.4]{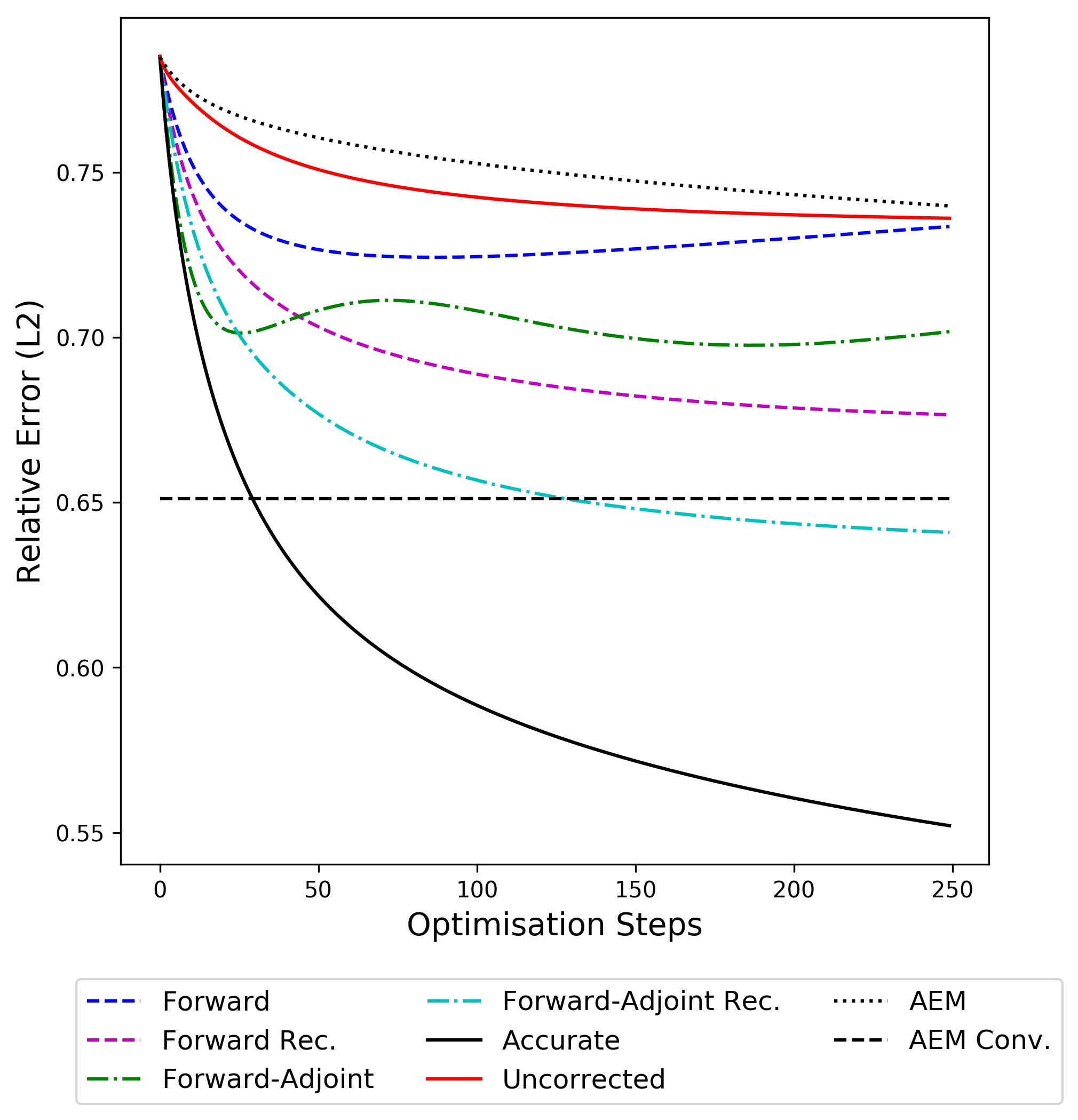}
		\caption{Relative reconstruction error (L2) for all methods on the vessel test set with 64 samples, tracked throughout the gradient descent scheme. \revision{250 steps of gradient descent were performed for all methods but AEM, where 20000 steps were taken.}}
		\label{fig:Quality_vessels}
	\end{minipage}%
	
\end{figure}

\begin{figure}[ht!]
	
	\centering
	\begin{subfigure}{\textwidth}
		\centering
		\caption{Reconstructions for first vessel phantom}
		\includegraphics[scale=0.4]{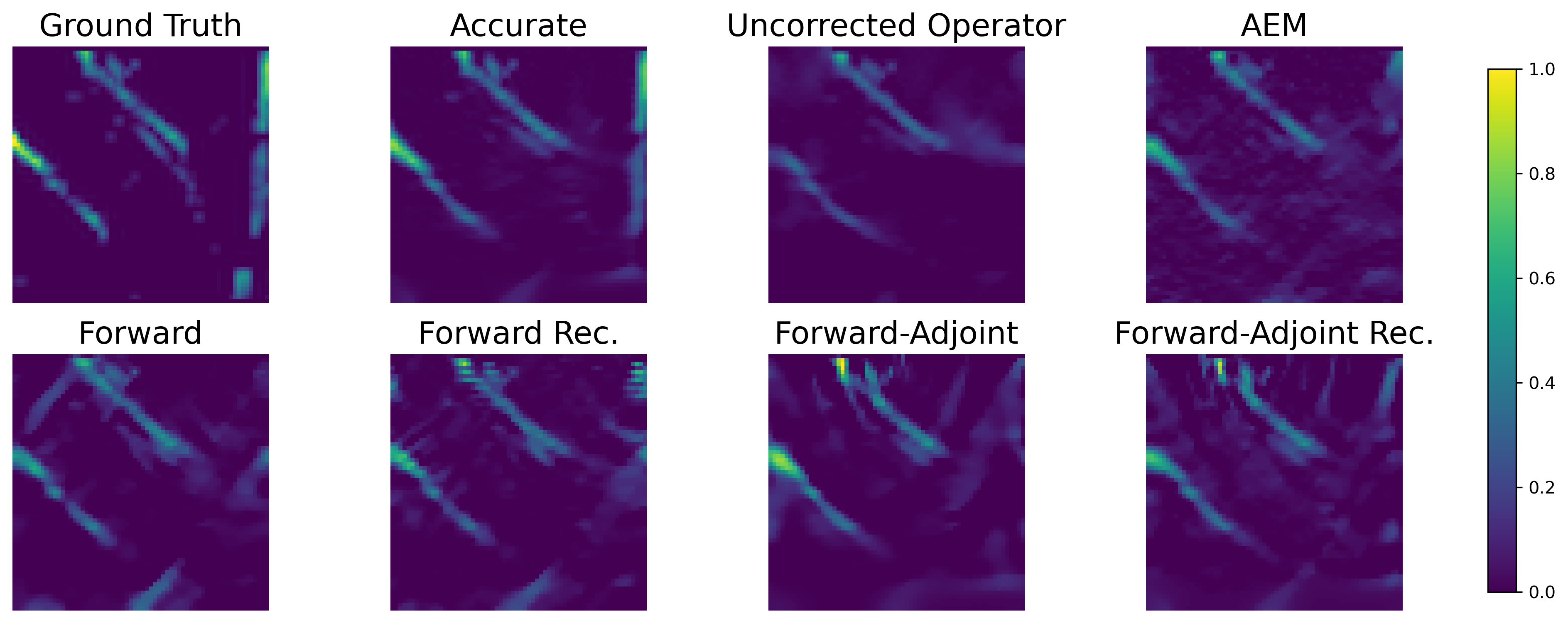}
	\end{subfigure}
	\begin{subfigure}{\textwidth}
		\centering
		\caption{Reconstructions for second vessel phantom}
		\includegraphics[scale=0.4]{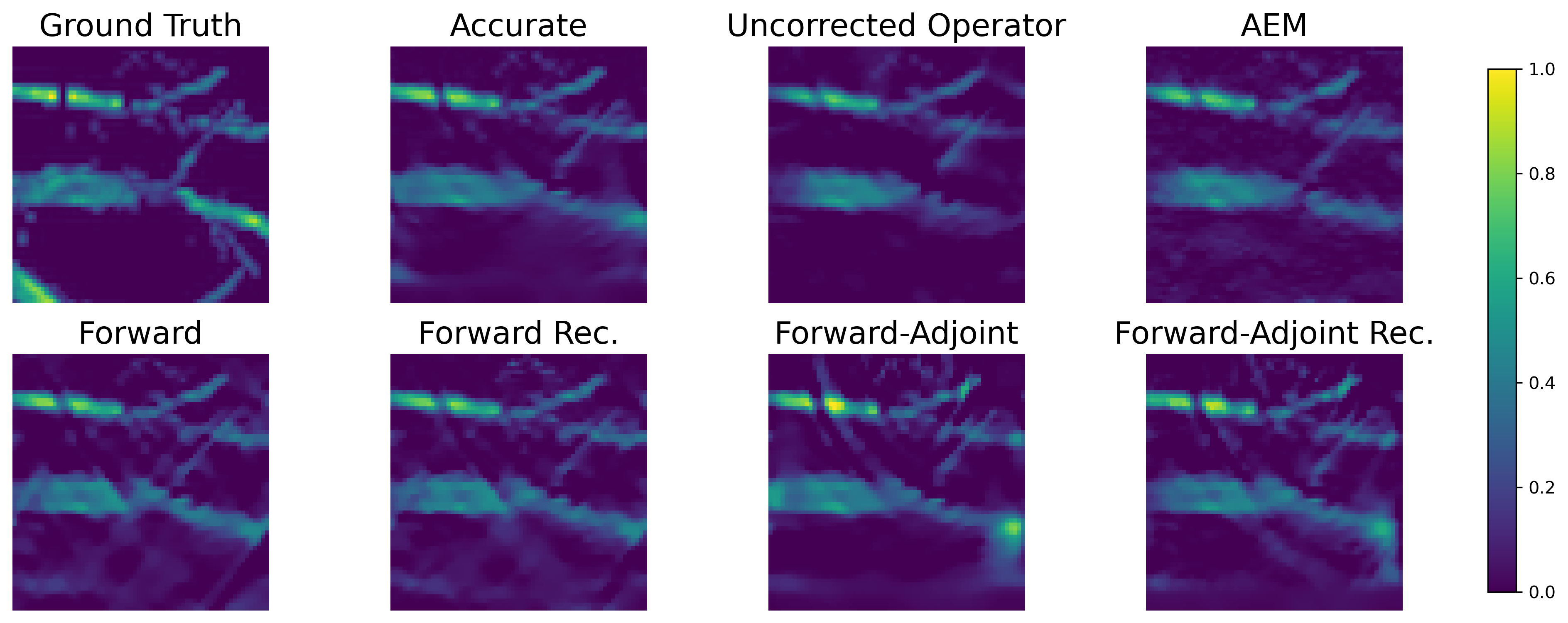}
	\end{subfigure}
	\caption{Reconstructions on vessels using the various operator corrections. We show the results after $250$ iterations of gradient descent for all methods but AEM, for which $20000$ iteration steps were taken. Huber regularisation is used.}
	\label{fig:Reconstructions_vessels}
\end{figure}

We present reconstructions for all discussed methods for two samples in Figure \ref{fig:Reconstructions_vessels}. We note for the first sample that the vessel structure at the right of the image completely disappears when using the uncorrected approximation. In fact, the corresponding measurement is severely reduced due to the thresholding of incident waves in the approximate model.
Hence, no correction method is able to fully recover the vessel structure at the right of the first sample, with AEM, forward method and forward-adjoint method coming closest. For all correction methods we observe a deterioration in reconstruction quality compared to the accurate operator. We note that the recursive forward method seems to lead to striping artefacts. Consistent with the quantitative results in Figure \ref{fig:Quality_vessels} the forward-adjoint recursive reconstructions are of the highest visual quality compared to the other reconstructions using a model corrections, leading to sharper results than the AEM baseline and to fewer artefacts than methods based on the forward only approach or those omitting recursive training. 
We remark that, up to some extent, perceived differences in smoothness can also be caused as the regularisation parameter has been optimised for all methods individually and hence might differ slightly between reconstructions.

To this end, we note that the training set with a total of $2760$ samples ($5520$ with rotations) is fairly small when taking into account the complexity of the vessel structures, see for instance the discussion with respect to AEM in  \cite{sahlstrom2020modeling}. It is hence possible that the remaining gap in reconstruction quality to the accurate operator could be closed further by using a more extensive training set. However, we expect that the gap cannot be closed completely on samples with a comparable complexity to the vessel phantoms as too much information might be lost in the thresholding step of the approximate operator that cannot be recovered even when taking into account the structure of the samples with highly parametrised learned corrections. This underlines the necessity of a statistical correction as discussed throughout Section \ref{sec:learnModelCor} to compensate for lost kernel directions in the approximate operator. 

\paragraph{Model transfer between vessel and ball phantoms}
\revision{
	In this paragraph we investigate how well the operator corrections trained on either the ball or the vessel samples generalise to the other of the two data sets. In particular, we discuss using models trained on balls to reconstruct vessels and vice versa. The aim of these experiments is to obtain a first understanding on how well trained model corrections generalise to new data sets in general, especially if the new set is very different from the training data in terms of image characteristics.
	\\
	When using models trained on the ball samples and tested on vessel images, we notice that the model gives reasonable corrections at initialisation of the variational scheme for the vessel samples, yielding corrected gradients. Nevertheless, the correction quality deteriorated rapidly during the gradient descent steps and the final reconstruction was not satisfactory compared to reconstructions obtained with the uncorrected approximate operator $\ApproxOp$. We hypothesise that the ball data was too distinct from the vessel samples and that the structure of the ball data was too simple for the learned model to perform reasonably on the much more complicated vessel data. In particular, the learned corrections were potentially fit very tightly to data and measurements induced by the ball phantoms that does not contain the same level of complexity as the vessel phantoms. Heuristically speaking, the data manifold of the ball samples seems to be too low-dimensional to generalise to other data.
	\\
	On the other hand, when using the forward-adjoint recursive model trained on the vessel samples on the ball samples, we obtained results that are clear improvements over reconstructions obtained with the uncorrected operator and are even comparable to the non-recursively trained methods on the ball data. We do, however, not match the performance of the forward-adjoint recursive model trained on the ball samples themselves. Figure \ref{fig:CrossData} shows reconstructions on a ball sample for various methods trained on the vessel samples. The reconstructions show a well-localised ball reconstruction with fairly sharp edges even in the challenging case of the ball sample located far from the detector plate. The results can be compared to results obtained with methods trained on the ball samples, as shown in Figure \ref{fig:Reconstructions_balls}. The visual assessment of reconstruction quality matches the quantitative results in terms of $L^2$ error as shown in Table \ref{tab:CrossData}.
	\\
	Finally, we note in both Figure \ref{fig:CrossData} and Table \ref{tab:CrossData} that adopting the regularisation parameter $\lambda$ of the forward-adjoint correction trained on vessel samples to a new optimal value for the ball data yields considerably improvements in performance. This demonstrates one of the main advantages of explicit corrections over their implicit counterparts, as separating between model correction and regularisation allows for an adaption of the regularisation parameter to the task, independently of the model correction learned.
}

\begin{figure}
	\centering
	\includegraphics[width=\textwidth]{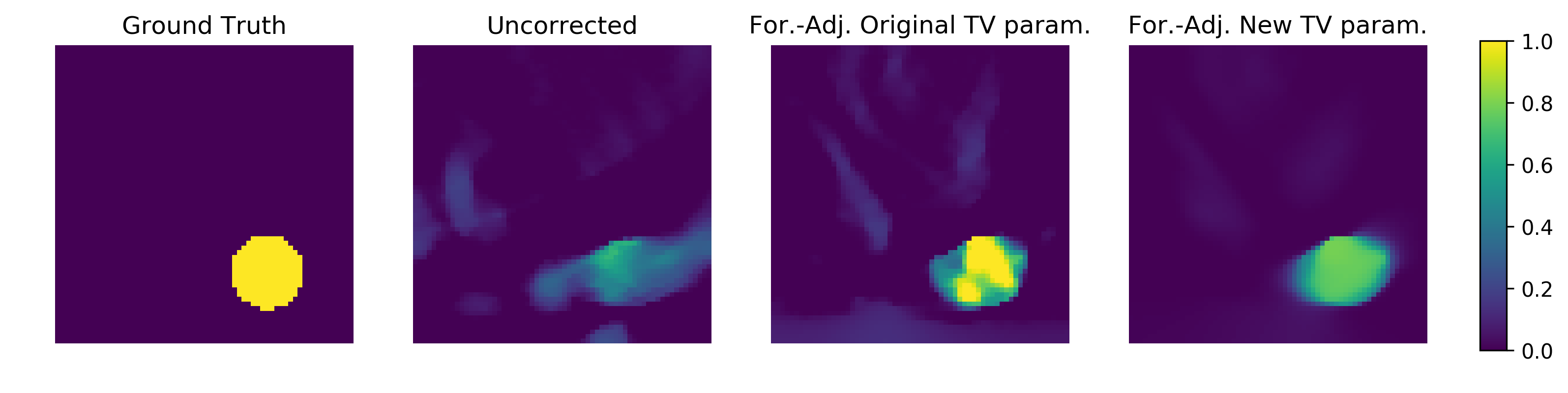}
	\caption{\revision{Models trained on vessel samples, evaluated on ball samples. From left to right: Ground truth image, reconstruction using the uncorrected operator, reconstruction using a recursive forward-adjoint correction with the same TV parameter as used on vessel data, reconstruction using a recursive forward-adjoint correction with new optimal TV parameter. 
	}}
	\label{fig:CrossData}
\end{figure}

\begin{table}[]
	\centering
	\begin{tabular}{l | c | c}
		& Training Data & $L^2$ error \\
		\toprule
		Accurate Operator & - & 0.11
		\\
		Approximate Operator & - & 0.55 \\
		\midrule
		Forward-Adjoint & balls & 0.15  \\
		\midrule
		For.-Adj. (old TV param.) & vessels & 0.40\\
		For.-Adj. (new TV param.) & vessels  & 0.35
	\end{tabular}
	\caption{\revision{Performance of the recursive forward-adjoint correction on ball samples. We evaluate the performance of models trained on vessel samples and compare to models trained on ball samples. Results are reported in terms of the $L^2$ error compared to the ground truth image.}}
	\label{tab:CrossData}
\end{table}






\section{Conclusion}\label{sec:Conclusions}

In this paper, we have introduced various approaches to learn a data-driven explicit model correction for inverse problems to be employed within a variational reconstruction framework. We have investigated several strategies to learn such a correction, starting with a simple forward correction for which we pointed out some fundamental limitations. In particular, we observed that this approach is limited by the range of the adjoint of the approximate operator when employed in a gradient descent scheme and is therefore unable to fully correct all modelling errors. To mitigate this, we have proposed a forward-adjoint correction as an alternative approach, overcoming these limitations by fitting an independent adjoint correction.

To ensure a model correction that can be employed throughout the optimisation process and avoid overfitting to the initial reconstruction, we proposed to augment all methods with a recursive training scheme. For the recursive forward-adjoint correction we provided a theoretical convergence analysis to show that the method approximates the accurate solution when trained to a sufficiently low loss. Finally, we have shown the potential of our approach on the task of limited-view photoacoustic tomography, demonstrating our theoretical considerations in practice and showing improved results compared to the commonly used AEM.

\revision{For the data chosen, the algorithm can be trained very quickly, requiring <2h for non-recursive experiments and around 16h for their recursive counterparts. For images larger than the $64 \times 64$ format used in the paper, the number of operations scales linearly with the number of pixels and hence quadratically with resolution in 2D and cubically in 3D. The actual increase in computational time might scale lower than the increase in operators as a higher amount of operations per layer increases the potential for parallelisation. The number of network parameters however does not necessarily change with resolution. Higher resolutions might make a more deep architecture appropriate, but the increase in weights caused by this would typically be strongly sublinear.}

This work is orthogonal to previous attempts at using neural networks to learn operator corrections, that were exclusively focused on the idea of implicit model corrections, learning the correction operator and a reconstruction prior simultaneously in an end-to-end trained reconstruction network. While this approach comes with advantages in terms of performance, our \textit{explicit} model correction allows to flexibly use any prior model alongside the corrected operator and can be integrated in the well-established framework of variational regularisation. Furthermore, our work unveils some of the challenges in model correction that are hidden in implicit schemes. Our findings can be used to inspire the design of novel implicit algorithms \revision{and allows for an analysis of implicit correction in future studies}. In particular, our observations on the limitations of the range of the adjoint of the approximation motivates the use of corrections in both reconstruction and data space for implicit model correction, motivating the use of algorithms such as learned primal-dual \cite{Adler2018}.

In future work one could apply the proposed method to different fields of application, such as computed tomography. In this application, the accurate model can be obtained by expensive photon-level Monte Carlo simulations, whereas a computationally efficient approximation is given by the widely-used ray transform. In general, applications to inverse problems involving non-linear operators are an interesting direction deserving further study, \revision{we refer to a related study exploring first ideas in this direction \cite{smyl2020learning}}. A class of very challenging applications are settings where we do not have explicit access to the accurate forward operator, but instead have access to empirical measurements only. Examples of such problems are tomography with slightly wrong estimated angles or deconvolution problems with errors in the point-spread function. These problems differ from the setting considered in this paper, where explicit access to the accurate operator was given and the approximation was performed to overcome computational constraints. In particular, the concept of recursive training, as presented here, requires explicit access to the accurate operator and is thus not readily applicable for problems where we have access to empirical measurements only, making them particularly challenging. \revision{We believe that in such settings, alternate training regimes that are not fully supervised and make use of secondary measures will be needed, estimating the approximation error from the data itself.}


Finally, we mention a possible combination of the proposed approach with AEM techniques. Since the latter, after training, yields a multi-variate Normal distribution as an estimate of the distribution of model errors it becomes increasingly unreliable as the non-Gaussianity of the accurate distribution increases. However, after an initial nonlinear correction of the form $\CorrectedOp$ described here, the AEM could be re-estimated using such a model. Commensurately, the estimated statistics of the model error from the AEM could be used in place of the simple $L^2$-loss used in the training in \eqref{eqn:ForwardLoss} and \eqref{eqn:ForwardLossFB} for example (i.e. the norm implied in the space $Y$). A possible future research direction could therefore be to iterate these approaches with a view to obtaining a more accurate probabilistic estimate of the eventual remaining model errors. 



\subsection*{Acknowledgements}
This work was partially supported by the Academy of Finland Projects 312123 and 312342 (Finnish Centre of Excellence in Inverse Modelling and Imaging, 2018--2025), Academy of Finland Projects 334817 and 314411, Jane and Aatos Erkko Foundation, as well as British Heart Foundation grant NH/18/1/33511, CMIC-EPSRC platform grant (EP/M020533/1), and EPSRC-Wellcome grant WT101957.

CBS acknowledges support from the Leverhulme Trust project on ‘Breaking the non-convexity barrier’, the Philip Leverhulme Prize, the EPSRC grants EP/S026045/1 and EP/T003553/1, the EPSRC Centre Nr. EP/N014588/1, the Wellcome Innovator Award RG98755, the RISE projects CHiPS and NoMADS, the Cantab Capital Institute for the Mathematics of Information and the Alan Turing Institute.

The work by SL was supported by the EPSRC grant EP/L016516/1 for the University of Cambridge Centre for Doctoral Training, the Cambridge Centre for Analysis and by the Cantab Capital Institute for the Mathematics of Information.

SA acknowledges support of EPSRC grants EP/N022750/1 and EP/T000864/1.

All authors acknowledge helpful discussions with Jonas Adler, Jari Kapio, Yury Korolev,  Ozan \"Oktem, and Peter Maass amongst others.

\bibliographystyle{siam}
\bibliography{Inverse_problems_refs}

\appendix

\section{An approximate model for photoacoustic tomography}\label{sec:appxPAT}
Here we discuss the accurate and approximate model as previously used in \cite{hauptmann2018MLMIR}. In photoacoustic tomography (PAT) a short pulse of near-infrared light is absorbed by chromophores in biological tissue. For a sufficiently short pulse, a the quantity of interest will result as a spatially-varying pressure increase $x$, which will initiate an 
ultrasound (US) pulse (\textit{photoacoustic effect}), that then propagates to the tissue surface. The measurement consists of the detected waves in space-time at the boundary of the tissue; this set of pressure time series constitutes the measured PA data $y$. 

For the forward model, this acoustic propagation is commonly modelled by an initial value problem for the wave equation \cite{Cox2005},
\begin{equation}
\label{eqn:PATfwd}
(\partial_{tt} - c^2 \Delta) p(\mathbf{x},t) = 0, \quad p(\mathbf{x},t = 0) = x(\mathbf{x}), \quad \partial_t p(\mathbf{x},t = 0) = 0, \text{ with } \mathbf{x}\in\R^2.
\end{equation}
The measurement is then modelled as a linear operator $\mathcal{M}$ acting on the pressure field $p(\mathbf{x},t)$ restricted to the boundary of the computational domain $\Omega$ and a finite time window:
\begin{equation}
y = \mathcal{M} \, p_{|\partial \Omega \times (0,T)}. \label{eqn:Measurement}
\end{equation}
Together, equations \eqref{eqn:PATfwd} and \eqref{eqn:Measurement} define the linear forward model that we consider in this study
\begin{equation}
Ax=y,
\label{eqn:Axeqy}
\end{equation}
from initial pressure $x$ to the measured time series $y$. This accurate forward model can be simulated by a pseudo-spectral time-stepping model as outlined in \cite{treeby2010,Treeby2012}. 

For the approximate model, we can exploit the fact that in our case the measurement points lie on a line ($\mathbf{x}_2=0$) outside the support of $x$, the pressure there can be related to $x$ by \cite{Cox2005,Koestli2001}:
\begin{align}
p(\mathbf{x}_1,t) = \frac{1}{c^2} 
\mathcal{F}_{k_1} \left\{\mathcal{C}_{\omega}\left\{
B({k_1},\omega) \tilde{x}({k_1},\omega)
\right\}\right\}, 
\label{eqn:FastFwd}
\end{align}
where 
$\tilde{x}({k_1},\omega)$ is obtained from $\hat{x}(k)$ via the dispersion relation $(\omega/c)^2 = k_1^2+k_2^2$ and 
$\hat{x}(k) = \mathcal{F}_{\mathbf{x}}\{x(\mathbf{x})\}$ is the 2D Fourier transform of $x(\mathbf{x})$. $\mathcal{C}_{\omega}$ is a cosine transform from $\omega$ to $t$, $\mathcal{F}_{k_1}$ is the 1D inverse Fourier Transform from $k_1$ to $\mathbf{x}_1$ on the detector. The weighting factor,
\begin{align}
B(k_1,\omega) = \omega/\left(\textrm{sgn}(\omega)\sqrt{(\omega/c)^2 - k_1^2}\right),
\end{align}
contains an integrable singularity which means that if \eqref{eqn:FastFwd} is evaluated by discretisation on a rectangular grid and thus enabling the application of FFT for efficient calculation), then aliasing in the measured data $p(\mathbf{x}_1,t)$ results. Consequently, evaluating \eqref{eqn:FastFwd} using FFT leads to a \emph{fast but approximate} forward model. In fact, we can control the degree of aliasing, by avoiding the singularity, that means in practice all components of $B$ for which $k_1^2 > (\omega/c)^2\sin^2\theta_{\max}$ are set to zero. This is equivalent to assuming only waves arriving at angles up to $\theta_{\max}$ from normal incidence are detected. 
We note, that there is a trade-off: the greater the range of angles included, the greater the aliasing. 
Finally, this results in a thresholded weighting factor $\widetilde{B}$ and hence the relation \eqref{eqn:FastFwd} using $\widetilde{B}$ defines the approximate model for this study: $\ApproxOp x = y$.

\section{Addition to theoretical results}\label{sec:appxTheo}

In this section, we only investigate the question of closeness of minimisers, without investigating if the minimisers of $\mathcal{L}_\Theta$ - that involves a nonlinear operator in the data term - can be identified efficiently using a gradient descent based algorithm. To answer this question, we will assume that the learned corrected operator $\CorrectedOp$ approximates the ground truth operator $\TrueOp$ sufficiently well, uniformly on some manifold $\mathcal{D}$ that contains the minimiser of $\mathcal{L}$. These assumptions represent the situation of a well-fit forward approximation on the data manifold $\mathcal{D}$ that we assume all relevant reconstructions to lie on. 

\revision{While it is difficult to check these assumptions in practice, the purpose of this discussion is to give a more complete theoretical view at the problem at hand, demonstrating that under sufficient assumptions closeness of forward operators is sufficient to deduce closeness of minimisers. However, this does not guarantee that the minimum can be found with a gradient descent algorithm or that a gradient descent algorithm even stays on the manifold $\mathcal{D}$ of good approximation quality. As a theoretical underpinning for the experiments conducted in this paper, Theorem \ref{thm:GradientConvergence} should hence instead be considered as a the main theorem.}

\begin{proposition}[Proximity of minimisers]\label{prop:proxMini}
	Denote by $\mathcal{D} \subset X$ the manifold of possible reconstructions that the operator approximation was trained on using empirical risk minimisation \eqref{equ:forwardLoss_FBC}. Let $\mathcal{L}$ be strongly convex.
	\revision{Assume further that $\mathcal{D}$ and the measurement noise is bounded. Hence for any $y = \TrueOp x_1 + \epsilon$ for some $x_1 \in \mathcal{D}$ and for any $x_2 \in \mathcal{D}$ 
		we have $\|Ax_2-y\|_Y \leq C$ (boundedness).}
	Let $\epsilon>0$. Denote by $\delta$ the corresponding quantity as in lemma \ref{lem:LowEnergyCloseness},
	without loss of generality let \revision{$\delta \leq 32C^2$}. Assume further $\CorrectedOp$ has been trained such that $\sup_{x \in \mathcal{D}}\|\CorrectedOp(x) - \TrueOp x\|_Y \leq \delta/4C$. Denote by
	\begin{align*}
	\hat{x} := \argmin_{x} \mathcal{L}(x), \quad \hat{x}_\Theta \in \argmin_{x} \mathcal{L}_\Theta(x)
	\end{align*}
	the reconstructions computed via the variational problem using either the accurate operator $\TrueOp$ or the corrected operator $\CorrectedOp$, respectively. Note that the \revision{minimiser} is unique for the functional $\mathcal{L}$ by strong convexity, but not necessarily for the functional $\mathcal{L}_\Theta$.
	Then for any $y \in B_r(0)$ that is such that both $\hat{x}, \hat{x}_\Theta \in \mathcal{D}$, we have
	\begin{align*}
	\| \hat{x} - \hat{x}_\Theta \|_X < \delta.
	\end{align*}
\end{proposition}

\begin{proof}
	First note that $|\mathcal{L}_\Theta(x) - \mathcal{L}(x)| \leq \delta/2$ for any $x \in \mathcal{D}$, as 
	\begin{align*}
	|\mathcal{L}_\Theta(x) - \mathcal{L}(x)|  =& \revision{\frac{1}{2}} \left| \|\CorrectedOp(x) - y\|_Y^2 - \|\TrueOp x - y \|_Y^2 \right|
	\\
	\leq& \| \TrueOp x - y \|_Y \|\TrueOp x - \CorrectedOp(x) \|_Y +  \frac{1}{2}\|\TrueOp x - \CorrectedOp(x) \|_Y^2
	\\
	\leq& \revision{C \cdot \frac{\delta}{4C} + \delta \frac{1}{2}\frac{32C^2}{(8C)^2}} = \delta/2
	\end{align*}
	By taking the minimum, this in particular implies that 
	\begin{align*}
	|\mathcal{L}_\Theta(\hat{x}_\Theta) - \mathcal{L}(\hat{x})| \leq \delta/2.
	\end{align*}
	We conclude via
	\begin{align*}
	|\mathcal{L}(\hat{x}_\Theta) - \mathcal{L}(\hat{x})| 
	\leq \
	|\mathcal{L}(\hat{x}_\Theta) - \mathcal{L}_\Theta(\hat{x}_\Theta)| + |\mathcal{L}_\Theta(\hat{x}_\Theta) - \mathcal{L}(\hat{x})| 
	\leq \ \delta/2 + \delta/2 = \delta,
	\end{align*}
	which finishes the proof using eq.(\ref{eqn:VaritionalAssumptionLoss}).
\end{proof}

\begin{remark}
	The assumption that $y$ is such that  $\hat{x}, \hat{x}_\Theta \in \mathcal{D}$ can be interpreted as a necessity for $y$ to have emerged from an underlying image that is close to the manifold of reconstructions $\mathcal{D}$ that the correction $\CorrectedOp$ has been trained on. Put differently, we require $y$ to be an actual realistic measurement, similar to those used to train the model correction $\CorrectedOp$.
\end{remark}


\end{document}